\documentclass[10pt,reqno]{amsart}
\usepackage{amsmath,amscd,amssymb}
\usepackage{url}
\usepackage{color}
\usepackage{tikz-cd}
\usepackage{subfigure}
\usepackage{enumerate}
\usepackage[pagebackref,colorlinks,citecolor=blue,linkcolor=magenta]{hyperref}

\usepackage[linesnumbered,ruled]{algorithm2e}
\RequirePackage{amsthm,amsmath,amsfonts,amssymb}
\usepackage[utf8]{inputenc}
\usepackage{chngcntr}
\usepackage{float}

\theoremstyle{plain}
   \newtheorem{theorem}{Theorem}[section]
   \newtheorem{proposition}[theorem]{Proposition}
   \newtheorem{lemma}[theorem]{Lemma}
   \newtheorem{corollary}[theorem]{Corollary}

\theoremstyle{definition}
   \newtheorem{definition}{Definition}[section]

   \newtheorem{example}{Example}[section] 
\theoremstyle{remark}
 \newtheorem{remark}{Remark}[section]

\newcommand{\xx}{\mathbf{x}}
\newcommand{\yy}{\mathbf{y}}

\newcommand{\LL}{\mathcal{L}}
\newcommand{\GG}{\mathcal{G}}

\newcommand{\MM}{\mathcal{M}}

\newcommand{\CC}{\mathcal{C}}
\newcommand{\RR}{\mathcal{R}}
\newcommand{\TT}{\mathcal{T}}

\newcommand{\R}{\mathbb{R}}

\newcommand{\toric}{\mathsf{toric}}

\SetKwInput{KwInput}{Input}
\SetKwInput{KwOutput}{Output}

\newcommand\independent{\protect\mathpalette{\protect\independenT}{\perp}}
\def\independenT#1#2{\mathrel{\rlap{$#1#2$}\mkern2mu{#1#2}}}

\def\newop#1{\expandafter\def\csname #1\endcsname{\mathop{\rm
#1}\nolimits}}

\newop{Inv}
\newop{conv}
\newop{pa}
\newop{de}
\newop{an}
\newop{nd}
\newop{ch}
\newop{CS}

\begin{document}
\title[A new characterization of decomposable models]{A new characterization of discrete decomposable graphical models}

\date{\today}

\author{Eliana Duarte}
\address{Otto-von-Guericke Universit\"at Magdebug, Universit\"atsplatz 2, 39106 Magdeburg, Germany}
\email{eliana.duarte@ovgu.de}

\author{Liam Solus}
%\date{\today}
\address{Institutionen f\"or Matematik, KTH, SE-100 44 Stockholm, Sweden}
\email{solus@kth.se}

\begin{abstract}
Decomposable graphical models, also known as perfect DAG models, play a fundamental role in standard approaches to probabilistic inference via graph representations in modern machine learning and statistics.  
However, such models are limited by the assumption that the data-generating distribution does not entail strictly context-specific conditional independence relations.  
The family of staged tree models generalizes DAG models so as to accommodate context-specific knowledge.  
We provide a new characterization of perfect discrete DAG models in terms of their staged tree representations.  
This characterization identifies the family of balanced staged trees as the natural generalization of discrete decomposable models to the context-specific setting.
\end{abstract}

%%---Subject Classifications-----------
%

%%---Keywords-----
%

\maketitle
\thispagestyle{empty}

%---SECTION: Introduction------------
\section{Introduction}
\label{sec: introduction}
A graphical model is a collection of joint probability distributions of a vector of random variables $(X_1,\ldots, X_p)$ that satisfy conditional independence constraints specified by a graph with set of nodes $[p]:=\{1,\ldots, p\}$.  Within the class  of graphical models there are those whose underlying graph is undirected and those whose underlying graph
is a directed acyclic graph (DAG). Both classes of models are usually defined in two different ways, either parametrically by means of a clique factorization for undirected models and the recursive factorization property for directed acyclic models, or implicitly via the undirected or directed Markov properties associated to each graph.
The class of models for which these two perspectives coincide is the set of \emph{decomposable graphical models}. These models have several different equivalent combinatorial,  statistical and geometric characterizations. In combinatorial terms, a model is decomposable if and only if
its underlying undirected graph is chordal or, equivalently, its  underlying DAG is \emph{perfect} \cite{L96}. In statistics, decomposable models are the subclass of
DAG models that are linear exponential families \cite{GHKM01}. In terms of algebra and geometry, discrete decomposable models
are toric varieties defined by quadratic binomials which correspond to global separation statements in \cite{GMS06}.

% \begin{theorem}
% \cite{GMS06}
% Let $\GG$ be an undirected graphical model for discrete variables.
% Then the following five statements are equivalent:
% \begin{itemize}
%     \item[(a)] $\GG$ is a decomposable graphical model.
%     \item[(b)] A distribution $P$ factors according to $\GG$ if and only if
%     $(P)$ satisfies a set of quadratic binomials corresponding to global
%     separation statements in $\GG$.
%     \item[(c)] The ideal defining $\MM(\GG)$ has a quadratic Gr\"obner basis in
%     which each polynomial corresponds to a global separation statement in $\GG$.
%     \item[(d)] The maximum likelihood estimate for $\GG$ is a rational function
%     \item[(e)] The set $\mathrm{image}(\phi_{A})$ is closed.
% \end{itemize}
% \end{theorem}
The main theorem of this paper, Theorem~\ref{thm: perfectly balanced and stratified}, provides a novel combinatorial characterization of discrete decomposable graphical models using staged tree models. Staged tree models are statistical models
that encode context-specific conditional independence relations among events by the use of a combinatorial object called a \emph{staged tree} \cite{SA08}. We show that a discrete DAG model is  decomposable if and only if its staged tree representation is \emph{balanced} \cite{DG20}. 
This result translates the combinatorial condition of a perfect DAG to staged trees and establishes balanced staged tree models as a natural generalization of discrete decomposable DAG models to the context-specific setting.

%---SECTION: DAG Models------
\section{Staged trees and DAG Models}
\label{sec: staged trees for DAGs}
The class of staged tree models was first introduced in \cite{SA08}.
We refer the reader to \cite{CGS18} for a detailed introduction to this model class.
Given a directed graph $\GG = (V,E)$ with node set $V$ and collection of edges $E$, for $v\in V$ we call $w\in V$ a \emph{parent} of $v$ (in $\GG$) if $w\rightarrow v\in E$, and we let $\pa_\GG(v)$ denote the collection of all parents of $v$ in $\GG$.  
Conversely, $v$ is a called a \emph{child} of $w$ (in $\GG$), and we let $\ch_\GG(w)$ denote the collection of all children of $w$ in $\GG$.  
A \emph{rooted tree} $\TT = (V,E)$ is a directed graph whose skeleton (i.e., underlying undirected graph) is a tree containing a unique \emph{root} node, $r$, for which $\pa_\TT(r) = \emptyset$.  
It follows that for each $v\in V$ there is a unique directed path from $r$ to $v$ in $\TT$, which we denote by $\lambda(v)$; that is, $\lambda(v)$ denotes the collection of all edges constituting this unique directed path.   
We also denote the set of edges from a node $v$ to each of its children
by $E(v)$.  

% A \emph{rooted tree} $\TT = (V,E)$ is a directed graph whose skeleton is a tree and for which there exists a unique node $r$, called the \emph{root} of $\TT$, whose set of \emph{parents} $\pa_\TT(r):=\{k\in V : k\rightarrow r\in E\}$ is empty.  For $v\in V$, let $\lambda(v)$ denote the unique directed path from $r$ to $v$ in $\TT$.
% We also let 
% $
% E(v) := \{ v\rightarrow w \in E : w\in \ch_{\TT}(v)\},
% $
% where $\ch_\TT(v):=\{k\in V : v\rightarrow k\in E\}$ is the set of \emph{children} of $v$.

%---DEFINITION: Staged Tree---
\begin{definition}
\label{def: staged tree}

For a rooted tree $\TT = (V,E)$, a finite set of labels $\LL$, and a map $\theta: E \rightarrow \LL$ labeling the edges $E$ with elements of $\LL$, the pair $(\TT,\theta)$ is a \emph{staged tree} if 
\begin{enumerate}
    \item $|\theta(E(v))| = |E(v)|$ for all $v\in V$, and
    \item  for any two $v,w\in V$, $\theta(E(v))$ and $\theta(E(w))$ are either equal or disjoint. 
\end{enumerate}
\end{definition}
We typically refer to $\TT$ as a staged tree whenever the labeling $\theta$ is understood. 
% When the labeling $\theta$ is understood, we may refer to $\TT$ as a staged tree. 
The second condition of Definition~\ref{def: staged tree} partitions the vertices of $\TT$ into disjoint sets, called \emph{stages}, defined by the property that $v,w\in V$ are in the same stage if and only if $\theta(E(v)) = \theta(E(w))$. 
The partition of $V$ into its stages is the \emph{staging} of $\TT$. 
% When we depict a staged tree, such as in 
% Figure~\ref{fig1A}, we will color two nodes the same color to indicate that they are in the same stage, with the convention that any white nodes are in a stage with cardinality one.
The space of canonical parameters of a staged tree $\TT$ is the set
\begin{equation*}
\label{eqn: parameter space}
\Theta_\TT := 
\left\{ \alpha\in\R^{|\LL|} :  \forall e\in E, \alpha_{\theta(e)}\in(0,1) \text{ and } \forall v \in V, \sum_{e\in E(v)}\alpha_{\theta(e)} = 1
\right\}.
\end{equation*}

\begin{definition}
\label{def: staged tree model}
Let ${\bf i}_\TT$ be the collection of all leaves of $\TT$.
 The 
 \emph{staged tree model} $\MM_{(\TT,\theta)}$ for $(\TT,\theta)$ is the image 
of the map $\psi_{\TT} : \Theta_\TT \longrightarrow \Delta^\circ_{|{\bf i}_\TT| -1}$ where
\begin{equation*}
\begin{split}
\psi_{\TT} &: \alpha \longmapsto f_v :=\left(\prod_{e\in E(\lambda(v))}\alpha_{\theta(e)}\right)_{v\in{\bf i}_\TT}, 
\end{split}
\end{equation*}
and $\Delta_{|{\bf i}_\TT|-1}^\circ$ denotes the $(|{\bf i}_\TT|-1)$-dimensional (open) probability simplex.
We say that $f\in\Delta_{|{\bf i}_{\TT}|-1}^\circ$ \emph{factorizes} according to $\TT$ if $f\in\MM_{(\TT,\theta)}$. 
% Two staged trees $\TT$ and $\TT^\prime$ are called \emph{statistically equivalent} if $\MM_{(\TT,\theta)} = \MM_{(\TT^\prime,\theta^\prime)}$.  
\end{definition}

Typically, we will identify the leaves of a staged tree $\TT$ with the possible outcomes of some jointly distributed random variables -- which is why the model $\MM_{(\TT,\theta)}$ is defined to live in the probability simplex $\Delta_{|{\bf i}_\TT| -1}^\circ$.    
To make this identification formal, let $X_{[p]} = (X_1,\ldots,X_p)$ denote a vector of discrete random variables with joint state space $\RR$.  
Given a subset $S\subset [p]$, we let $\RR_S$ denote the restricted state space of the random subvector $X_S = (X_i : i\in S)$ of $X_{[p]}$.  
% Throughout this paper we will work with trees that
% represent the outcome space of the vector of discrete 
% random variables $X_{[p]}$ as a sequence of events.
For a permutation $\pi_1\cdots\pi_p\in\mathfrak{S}_p$ of $[p]$, construct a rooted tree $\TT = (V,E)$ where $V := \{r\}\cup\bigcup_{j\in[p]}\RR_{\{\pi_1,\ldots,\pi_j\}}$, and $E$ is the set
\[
         \{r\rightarrow x_{\pi_1} : x_{\pi_1}\in\RR_{\{\pi_1\}}\}\cup\{x_{\pi_1}\cdots x_{\pi_{k-1}}  \rightarrow x_{\pi_1}\cdots x_{\pi_{k}}  : x_{\pi_1}\cdots x_{\pi_{k}}\in\RR_{\{\pi_1,\ldots,\pi_k\}}\}.
\]
For $v\in V$, we call the number of edges in $\lambda(v)$ the \emph{level} of $v$, and for $k\in\{0,\ldots, p\}$, we refer to the set of all nodes with level $k$, denoted $L_k$, as the $k^{th}$ \emph{level} of $\TT$. 
% The \emph{level} of a node $v\in V$ is the number of edges in $\lambda(v)$, and for $k\in\{0,\ldots, p\}$, the $k^{th}$ \emph{level} of $\TT$, denoted $L_k$, is the collection of all nodes $v\in V$ with level $k$. 
As the root node is the only element in $L_0$, we typically ignore this level. 
% Note that $L_0 = \{r\}$, so we usually ignore this level when referring to the levels of $\TT$.  
For trees defined from a vector of random variables $X_{[p]}$ as above, the $k^{th}$ level of $\TT$, for $k>0$, is simply the set of outcomes $\RR_{\{\pi_1,\ldots,\pi_k\}}$. 
Hence, we associate the variable $X_{\pi_k}$ with level $L_k$ and denote this association by $(L_1,\ldots,L_k)\sim(X_{\pi_1},\ldots,X_{\pi_p})$.  
We call the permutation $\pi$ the \emph{causal ordering} of $\TT$. 
% Given such a tree $\TT$ with levels $(L_1,\ldots,L_k)\sim(X_{\pi_1},\ldots,X_{\pi_p})$, we call the permutation $\pi = \pi_1\cdots\pi_p$ the \emph{causal ordering} of $\TT$.  

Note that a tree $\TT$ defined for the vector $X_{[p]}$ is \emph{uniform}; that is, $|E(v)| = |E(w)|$ for any $v,w\in L_k$, for all $k\geq 0$. 
They are also \emph{stratified}; meaning that all of their leaves have the same level, and if any two nodes are in the same stage then they are also in the same level.
For a uniform and stratified staged tree model $\MM_{(\TT,\theta)}$, the parameter values on the edges $E(v)$ for $v\in \RR$ abide by the chain rule: % (see \cite[Lemma 3.2]{DS20}).  
%---LEMMA: Parameter Interpretation---
\begin{lemma}
\label{lem: parameter interpretation}
Let $\TT$ be a uniform, stratified staged tree with levels $(L_1,\ldots,L_p)\sim(X_1,\ldots,X_p)$, let $f\in\MM_{(\TT,\theta)}$, and fix $v\in{\bf i}_\TT$.  
If $e\in \lambda(v)$ is the edge $u\rightarrow w$ between levels $L_{k-1}$ and $L_k$, where $u = x_1\cdots x_{k-1}$ and $w =  x_1\cdots x_{k-1}x_k$, then 
\[
\alpha_{\theta(e)} = f(x_k \mid x_1\cdots x_{k-1}).
\]
\end{lemma}

\begin{proof}
The proof is by induction on $p$, the number of levels of $\TT$.  
For the base case, let $p = 1$.  
Given $v\in{\bf i}_\TT$, we have that $f_v = \alpha_{\theta(e)} = f(X_1 = x_1)$, where $v = x_1\in\RR_{\{1\}}$.  
Suppose now that the claim holds for $p-1$ for some $p>1$.  
Then, $f_v = f(x_1\cdots x_p)$, where $v= x_1\cdots x_p$, and by definition of $f_v$ (via Definition~\ref{def: staged tree model}), we have
\begin{equation}
\label{eqn:lemprod}
f_v = \prod_{e\in E(\lambda(v))}\alpha_{\theta(e)}.
\end{equation}
Since $v= x_1\cdots x_p$, then by the chain rule
\begin{equation}
\label{eqn:chain}
f_v = f(x_1\cdots x_p) = \prod_{k \in[p]} f( x_k \mid x_1\cdots x_{k-1}).
\end{equation}
Let $e$ denote the unique edge in $\lambda(v)$ between a node in level $L_{p-1}$ of $\TT$ and a node in level $L_p$.  
By the inductive hypothesis, the equality of equations~\eqref{eqn:lemprod} and~\eqref{eqn:chain} reduces to 
$
\alpha_{\theta(e)} = f(x_p \mid x_1\cdots x_{p-1}),
$
which completes the proof.
\end{proof}

% Namely, if $f\in\MM_{(\TT,\theta)}$ then for any $v = x_{\pi_1}\cdots x_{\pi_p}\in\RR$ and $e = x_{\pi_1}\cdots x_{\pi_{k-1}}\rightarrow x_{\pi_1}\cdots x_{\pi_k}$ in $\lambda (v)$, it holds that
%  $x_{\theta(e)} = f(x_{\pi_k} \mid x_{\pi_1}\cdots x_{\pi_{k-1}})$.
A uniform, stratified staged tree $\TT$ with levels $(L_1,\ldots,L_k)\sim(X_{\pi_1},\ldots,X_{\pi_p})$ is called \emph{compatibly labeled} if
\[
\theta(x_{\pi_1}\cdots x_{\pi_{k-1}} \rightarrow x_{\pi_1}\cdots x_{\pi_{k-1}}x_{\pi_k}) = \theta(y_{\pi_1}\cdots y_{\pi_{k-1}} \rightarrow y_{\pi_1}\cdots y_{\pi_{k-1}}x_{\pi_k})
\]
for all $x_{\pi_k}\in\RR_{\{\pi_k\}}$ whenever $x_{\pi_1}\cdots x_{\pi_{k-1}}$ and $y_{\pi_1}\cdots y_{\pi_{k-1}}$ are in the same stage.  
This condition ensures that edges emanating from $x_{\pi_1}\cdots x_{\pi_{k-1}}$ and $y_{\pi_1}\cdots y_{\pi_{k-1}}$ with endpoints corresponding to the same outcome $x_{\pi_{k}}$ of the next variable encode the invariance in conditional probabilities: 
\begin{equation}
    \label{eqn: invariances}
f(x_{\pi_k} \mid x_{\pi_1}\cdots x_{\pi_{k-1}}) = f(x_{\pi_k} \mid y_{\pi_1}\cdots y_{\pi_{k-1}}).
\end{equation}
These invariances, when considered in terms of the staging of the tree, collectively encode context-specific conditional independence relations satisfied by distributions in the model $\MM_{(\TT,\theta)}$.  
The most relevant examples of this are DAG models.

%---SUBSECTION: Staged Trees Associated to DAG models---
\subsection{Staged tree representations of DAG models}
\label{subsec: staged trees for DAGs}
Given a DAG $\GG = ([p],E)$ and a vector $X_{[p]}$ with state space $\RR$,  the \emph{DAG model} associated to $\GG$, denoted $\MM(\GG)$, is the set of all distributions $f\in\Delta_{|\RR|-1}^\circ$ that satisfy the recursive factorization
\begin{equation}
\label{eqn: DAG factorization}
f(X)=\prod_{k=1}^p f(X_k \mid X_{\pa_{\GG}(k)}).
\end{equation}
A DAG model $\MM(\GG)$ is called a \emph{decomposable model} if the DAG $\GG$ is \emph{perfect}; i.e., for all $k\in[p]$ the induced subDAG of $\GG$ on $\pa_\GG(k)$ is complete.
To derive our characterization of decomposable graphical models in terms of  their associated staged trees, we first characterize those staged trees that represent DAG models.  

A permutation $\pi = \pi_1\cdots\pi_p\in\mathfrak{S}_p$ of $[p]$ is called a \emph{linear extension} (or \emph{topological ordering}) of a DAG $\GG = ([p],E)$ if $\pi^{-1}_i < \pi^{-1}_j$ whenever $i\rightarrow j\in E$.  
Given a linear extension $\pi$ of a DAG $\GG$, we construct a staged tree $\TT_\GG^\pi$ for the vector $X_{[p]}$ with causal ordering $\pi$ by labeling the edges emanating from level $L_k = \RR_{\{\pi_1,\ldots, \pi_k\}}$ as follows:
Let 
\[
\LL := \{f(x_{\pi_{k+1}} \mid x_{\pa_\GG(\pi_{k+1})}) : k\in[p-1], x_{\pi_{k+1}}\in\RR_{\{\pi_{k+1}\}}, x_{\pa_\GG(\pi_{k+1})}\in \RR_{\pa_\GG(\pi_{k+1})}\},
\]
and the labeling map $\theta: E \longrightarrow \LL$ where
\[
\theta(x_{\pi_1}\cdots x_{\pi_k}\rightarrow x_{\pi_1}\cdots x_{\pi_k}x_{\pi_{k+1}}) = f(x_{\pi_{k+1}} \mid x_{\pa_\GG(\pi_{k+1})}).
\]
It follows that $(\TT_\GG^\pi,\theta)$ is stratified and uniform, and hence, for any $f\in\MM_{(\TT_\GG^\pi,\theta)}$,
\[
f(x_{\pi_{k+1}} \mid x_{\pi_1}\cdots x_{\pi_k}) = \alpha_{\theta(e)} = \alpha_{f(x_{\pi_{k+1}} \mid x_{\pa_\GG(\pi_{k+1})})},
\] 
where $e = x_{\pi_1}\cdots x_{\pi_k}\rightarrow x_{\pi_1}\cdots x_{\pi_k}x_{\pi_{k+1}}$. 
Moreover, $\TT_\GG^\pi$ is compatibly labeled with a stage 
\[
S_{y_{\pa_\GG(k+1)}} := \{ x=x_{\pi_1}\cdots x_{\pi_k} \in L_k : x_{\{\pi_1,\ldots,\pi_k\}\cap\pa_\GG(k+1)} = y_{\pa_\GG(k+1)}\}
\]
in level $L_k$ for each $k\in[p-1]$ and each $y_{\pa_\GG(k+1)}\in\RR_{\pa_\GG(k+1)}$.
Hence, the invariances in equation~\eqref{eqn: invariances} implied by $\TT_\GG^\pi$ being compatibly labeled imply that
\[
f(x_{\pi_{k+1}} \mid x_{\pi_1}\cdots x_{\pi_k}) = \alpha_{f(x_{\pi_{k+1}} \mid x_{\{\pi_1,\ldots,\pi_{k}\}\cap\pa_\GG(k+1)})} =
\alpha_{f(x_{\pi_{k+1}} \mid x_{\pa_\GG(k+1)})}, 
\]
where the last equality follows from the assumption that $\pi$ is a linear extension of $\GG$.  
Thus, $\MM_{(\TT_\GG^\pi,\theta)}$ consists of all distributions $f\in\Delta_{|\RR|-1}^\circ$ that satisfy equation~\eqref{eqn: DAG factorization}; that is, $\MM(\GG) = \MM_{(\TT_\GG^\pi,\theta)}$.  
Hence, $\TT_\GG^\pi$ is a stage tree representation of a DAG model, which we call the staged tree \emph{associated to} $\GG$ and $\pi$. 

\begin{remark}
From Lemma~\ref{lem: parameter interpretation}, the value $f(x_{\pi_{k+1}} \mid x_{\pi_1}\cdots x_{\pi_k})$ indicates a transition probability
for any staged tree that is compatibly labeled with levels $(L_1,\ldots, L_p)\sim (X_{\pi_1},\ldots, X_{\pi_p})$. The notation for this transition probability is different from the  labels in a staged tree representation of a DAG model:
$\LL := \{f(x_{\pi_{k+1}} \mid x_{\pa_\GG(\pi_{k+1})}) : k\in[p-1], x_{\pi_{k+1}}\in\RR_{\{\pi_{k+1}\}}, x_{\pa_\GG(\pi_{k+1})}\in \RR_{\pa_\GG(\pi_{k+1})}\}$. In the former case the notation indicates a parameter in the latter
case the notation indicates a symbol. For the proof of the main theorem we will manipulate the elements
in $\LL$ as indeterminates.
\end{remark}

\begin{example}
Consider the DAGs $\GG_{1}=([4], \{1\to 3, 2\to 3, 3\to 4, 2\to 4\})$ and
$\GG_2=([4],\{1\to 2, 2\to 3 , 2\to 4, 3\to 4\})$.
The trees $\TT_1, \TT_2$ in Figure~\ref{fig:staged tree representation} are staged tree representations of the DAG models
$\MM(\GG_1)$ and $\MM(\GG_2)$ with respect to the linear extension $\pi=1234$   where $X_1,X_2,X_3,X_4$
are binary random variables. An upwards arrow in Figure~\ref{fig:staged tree representation} represents the
outcome $0$ and a downwards arrow represents the outcome $1$. Two vertices with the same color (except white) in any of the
trees represent a context-specific conditional independence relation. For instance, the four colours in level three
of $\TT_1,\TT_2$, respectively, represent  $f(X_4|(X_1,X_2,X_3)=(0,i,j))=f(X_4|(X_1,X_2,X_3)=(1,i,j))$ for each $i,j\in \{0,1\}$, with a different stage/color for each possible pair $i,j$. Together these four context-specific conditional independence relations combine to make the CI relation $X_4 \independent X_1 |(X_2,X_3)$.
\end{example}

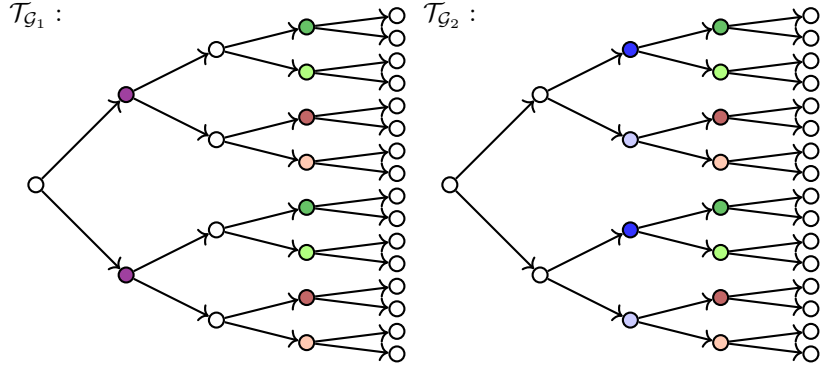
\begin{figure}
    \centering
     \begin{tikzpicture}[thick,scale=0.3]
	
	%---NODES---	 	
 	 \node[circle, draw, fill=black!0, inner sep=2pt, minimum width=1pt] (w3) at (0,0)  {};

 	 \node[circle, draw, fill=black!0, inner sep=2pt, minimum width=1pt] (w4) at (0,-1) {};
 	 \node[circle, draw, fill=black!0, inner sep=2pt, minimum width=1pt] (w5) at (0,-2) {};
 	 \node[circle, draw, fill=black!0, inner sep=2pt, minimum width=1pt] (w6) at (0,-3) {};
	 
 	 \node[circle, draw, fill=black!0, inner sep=2pt, minimum width=1pt] (v3) at (0,-4)  {};
 	 \node[circle, draw, fill=black!0, inner sep=2pt, minimum width=1pt] (v4) at (0,-5) {};
 	 \node[circle, draw, fill=black!0, inner sep=2pt, minimum width=1pt] (v5) at (0,-6) {};
 	 \node[circle, draw, fill=black!0, inner sep=2pt, minimum width=1pt] (v6) at (0,-7) {};

	 \node[circle, draw, fill=green!60!black!60, inner sep=2pt, minimum width=2pt] (w1) at (-4,-.5) {};
 	 \node[circle, draw, fill=green!60!yellow!50, inner sep=2pt, minimum width=2pt] (w2) at (-4,-2.5) {}; 

 	 \node[circle, draw, fill=red!60!black!60, inner sep=2pt, minimum width=2pt] (v1) at (-4,-4.5) {};
 	 \node[circle, draw, fill=red!70!yellow!30, inner sep=2pt, minimum width=2pt] (v2) at (-4,-6.5) {};

 	 \node[circle, draw, fill=green!0, inner sep=2pt, minimum width=2pt] (w) at (-8,-1.5) {};

 	 \node[circle, draw, fill=green!0, inner sep=2pt, minimum width=2pt] (v) at (-8,-5.5) {};

 	 \node[circle, draw, fill=violet!75, inner sep=2pt, minimum width=2pt] (r) at (-12,-3.5) {};

 	 \node[circle, draw, fill=black!0, inner sep=2pt, minimum width=2pt] (w3i) at (0,-8)  {};
 	 \node[circle, draw, fill=black!0, inner sep=2pt, minimum width=2pt] (w4i) at (0,-9) {};
 	 \node[circle, draw, fill=black!0, inner sep=2pt, minimum width=2pt] (w5i) at (0,-10) {};
 	 \node[circle, draw, fill=black!0, inner sep=2pt, minimum width=2pt] (w6i) at (0,-11) {};
	 
 	 \node[circle, draw, fill=black!0, inner sep=2pt, minimum width=2pt] (v3i) at (0,-12)  {};
 	 \node[circle, draw, fill=black!0, inner sep=2pt, minimum width=2pt] (v4i) at (0,-13) {};
 	 \node[circle, draw, fill=black!0, inner sep=2pt, minimum width=2pt] (v5i) at (0,-14) {};
 	 \node[circle, draw, fill=black!0, inner sep=2pt, minimum width=2pt] (v6i) at (0,-15) {};

	 \node[circle, draw, fill=green!60!black!60, inner sep=2pt, minimum width=2pt] (w1i) at (-4,-8.5) {};
 	 \node[circle, draw, fill=green!60!yellow!50, inner sep=2pt, minimum width=2pt] (w2i) at (-4,-10.5) {};

 	 \node[circle, draw, fill= red!60!black!60 , inner sep=2pt, minimum width=2pt] (v1i) at (-4,-12.5) {};
 	 \node[circle, draw, fill= red!70!yellow!30, inner sep=2pt, minimum width=2pt] (v2i) at (-4,-14.5) {};

 	 \node[circle, draw, fill=red!0, inner sep=2pt, minimum width=2pt] (wi) at (-8,-9.5) {};

 	 \node[circle, draw, fill=red!0, inner sep=2pt, minimum width=2pt] (vi) at (-8,-13.5) {};

 	 \node[circle, draw, fill=violet!75, inner sep=2pt, minimum width=2pt] (ri) at (-12,-11.5) {};

 	 \node[circle, draw, fill=black!0, inner sep=2pt, minimum width=2pt] (I) at (-16,-7.5) {};
 	 \node (T2) at (-16,0) {$\TT_{\GG_1}:$};
 	%-- DAG

	%---EDGES---	 
 	 \draw[->]   (I) --    (r) ;
 	 \draw[->]   (I) --   (ri) ;

 	 \draw[->]   (r) --   (w) ;
 	 \draw[->]   (r) --   (v) ;

 	 \draw[->]   (w) --  (w1) ;
 	 \draw[->]   (w) --  (w2) ;

 	 \draw[->]   (w1) --   (w3) ;
 	 \draw[->]   (w1) --   (w4) ;
 	 \draw[->]   (w2) --  (w5) ;
 	 \draw[->]   (w2) --  (w6) ;

 	 \draw[->]   (v) --  (v1) ;
 	 \draw[->]   (v) --  (v2) ;

 	 \draw[->]   (v1) --  (v3) ;
 	 \draw[->]   (v1) --  (v4) ;
 	 \draw[->]   (v2) --  (v5) ;
 	 \draw[->]   (v2) --  (v6) ;

 	 \draw[->]   (ri) --   (wi) ;
 	 \draw[->]   (ri) -- (vi) ;

 	 \draw[->]   (wi) --  (w1i) ;
 	 \draw[->]   (wi) --  (w2i) ;

 	 \draw[->]   (w1i) --  (w3i) ;
 	 \draw[->]   (w1i) -- (w4i) ;
 	 \draw[->]   (w2i) --  (w5i) ;
 	 \draw[->]   (w2i) --  (w6i) ;

 	 \draw[->]   (vi) --  (v1i) ;
 	 \draw[->]   (vi) --  (v2i) ;

 	 \draw[->]   (v1i) --  (v3i) ;
 	 \draw[->]   (v1i) -- (v4i) ;
 	 \draw[->]   (v2i) -- (v5i) ;
 	 \draw[->]   (v2i) --  (v6i) ;

\end{tikzpicture}
 \begin{tikzpicture}[thick,scale=0.3]

 	 \node[circle, draw, fill=black!0, inner sep=2pt, minimum width=1pt] (w3) at (0,0)  {};

 	 \node[circle, draw, fill=black!0, inner sep=2pt, minimum width=1pt] (w4) at (0,-1) {};
 	 \node[circle, draw, fill=black!0, inner sep=2pt, minimum width=1pt] (w5) at (0,-2) {};
 	 \node[circle, draw, fill=black!0, inner sep=2pt, minimum width=1pt] (w6) at (0,-3) {};
	 
 	 \node[circle, draw, fill=black!0, inner sep=2pt, minimum width=1pt] (v3) at (0,-4)  {};
 	 \node[circle, draw, fill=black!0, inner sep=2pt, minimum width=1pt] (v4) at (0,-5) {};
 	 \node[circle, draw, fill=black!0, inner sep=2pt, minimum width=1pt] (v5) at (0,-6) {};
 	 \node[circle, draw, fill=black!0, inner sep=2pt, minimum width=1pt] (v6) at (0,-7) {};

	 \node[circle, draw, fill=green!60!black!60, inner sep=2pt, minimum width=2pt] (w1) at (-4,-.5) {};
 	 \node[circle, draw, fill=green!60!yellow!50, inner sep=2pt, minimum width=2pt] (w2) at (-4,-2.5) {}; 

 	 \node[circle, draw, fill=red!60!black!60, inner sep=2pt, minimum width=2pt] (v1) at (-4,-4.5) {};
 	 \node[circle, draw, fill=red!70!yellow!30, inner sep=2pt, minimum width=2pt] (v2) at (-4,-6.5) {};

 	 \node[circle, draw, fill=blue!80, inner sep=2pt, minimum width=2pt] (w) at (-8,-1.5) {};

 	 \node[circle, draw, fill=blue!20, inner sep=2pt, minimum width=2pt] (v) at (-8,-5.5) {};

 	 \node[circle, draw, fill=violet!0, inner sep=2pt, minimum width=2pt] (r) at (-12,-3.5) {};

 	 \node[circle, draw, fill=black!0, inner sep=2pt, minimum width=2pt] (w3i) at (0,-8)  {};
 	 \node[circle, draw, fill=black!0, inner sep=2pt, minimum width=2pt] (w4i) at (0,-9) {};
 	 \node[circle, draw, fill=black!0, inner sep=2pt, minimum width=2pt] (w5i) at (0,-10) {};
 	 \node[circle, draw, fill=black!0, inner sep=2pt, minimum width=2pt] (w6i) at (0,-11) {};
	 
 	 \node[circle, draw, fill=black!0, inner sep=2pt, minimum width=2pt] (v3i) at (0,-12)  {};
 	 \node[circle, draw, fill=black!0, inner sep=2pt, minimum width=2pt] (v4i) at (0,-13) {};
 	 \node[circle, draw, fill=black!0, inner sep=2pt, minimum width=2pt] (v5i) at (0,-14) {};
 	 \node[circle, draw, fill=black!0, inner sep=2pt, minimum width=2pt] (v6i) at (0,-15) {};

	 \node[circle, draw, fill=green!60!black!60, inner sep=2pt, minimum width=2pt] (w1i) at (-4,-8.5) {};
 	 \node[circle, draw, fill=green!60!yellow!50, inner sep=2pt, minimum width=2pt] (w2i) at (-4,-10.5) {};

 	 \node[circle, draw, fill= red!60!black!60 , inner sep=2pt, minimum width=2pt] (v1i) at (-4,-12.5) {};
 	 \node[circle, draw, fill= red!70!yellow!30, inner sep=2pt, minimum width=2pt] (v2i) at (-4,-14.5) {};

 	 \node[circle, draw, fill=blue!80, inner sep=2pt, minimum width=2pt] (wi) at (-8,-9.5) {};

 	 \node[circle, draw, fill=blue!20, inner sep=2pt, minimum width=2pt] (vi) at (-8,-13.5) {};

 	 \node[circle, draw, fill=violet!0, inner sep=2pt, minimum width=2pt] (ri) at (-12,-11.5) {};

 	 \node[circle, draw, fill=black!0, inner sep=2pt, minimum width=2pt] (I) at (-16,-7.5) {};
 	 
 	  \node (T2) at (-16,0) {$\TT_{\GG_2}:$};

	%---EDGES---	 
 	 \draw[->]   (I) --    (r) ;
 	 \draw[->]   (I) --   (ri) ;

 	 \draw[->]   (r) --   (w) ;
 	 \draw[->]   (r) --   (v) ;

 	 \draw[->]   (w) --  (w1) ;
 	 \draw[->]   (w) --  (w2) ;

 	 \draw[->]   (w1) --   (w3) ;
 	 \draw[->]   (w1) --   (w4) ;
 	 \draw[->]   (w2) --  (w5) ;
 	 \draw[->]   (w2) --  (w6) ;

 	 \draw[->]   (v) --  (v1) ;
 	 \draw[->]   (v) --  (v2) ;

 	 \draw[->]   (v1) --  (v3) ;
 	 \draw[->]   (v1) --  (v4) ;
 	 \draw[->]   (v2) --  (v5) ;
 	 \draw[->]   (v2) --  (v6) ;

 	 \draw[->]   (ri) --   (wi) ;
 	 \draw[->]   (ri) -- (vi) ;

 	 \draw[->]   (wi) --  (w1i) ;
 	 \draw[->]   (wi) --  (w2i) ;

 	 \draw[->]   (w1i) --  (w3i) ;
 	 \draw[->]   (w1i) -- (w4i) ;
 	 \draw[->]   (w2i) --  (w5i) ;
 	 \draw[->]   (w2i) --  (w6i) ;

 	 \draw[->]   (vi) --  (v1i) ;
 	 \draw[->]   (vi) --  (v2i) ;

 	 \draw[->]   (v1i) --  (v3i) ;
 	 \draw[->]   (v1i) -- (v4i) ;
 	 \draw[->]   (v2i) -- (v5i) ;
 	 \draw[->]   (v2i) --  (v6i) ;

\end{tikzpicture}
%  \begin{tikzpicture}[thick,scale=0.3]
%  	%-- Nodes for the DAG
%  	\node[circle, draw, fill=black!0, inner sep=2pt, minimum width=2pt] (X1) at (-16,-18) {\footnotesize$1$};
%  	\node[circle, draw, fill=black!0, inner sep=2pt, minimum width=2pt] (X2) at (-12,-18) {\footnotesize$2$};
%  	\node[circle, draw, fill=black!0, inner sep=2pt, minimum width=2pt] (X3) at (-10,-16) {\footnotesize$3$};
%  	\node[circle, draw, fill=black!0, inner sep=2pt, minimum width=2pt] (X4) at (-8,-18) {\footnotesize$4$};
%  	\draw[->] (X1) -- (X2)
%  	\draw[->] (X2) -- (X4)
% \end{tikzpicture}
%  \begin{tikzpicture}[thick,scale=0.3]

%  	%-- Nodes for the DAG
%  	\node[circle, draw, fill=black!0, inner sep=2pt, minimum width=2pt] (X1) at (-16,-18) {\footnotesize$1$};
%  	\node[circle, draw, fill=black!0, inner sep=2pt, minimum width=2pt] (X2) at (-12,-18) {\footnotesize$2$};
%  	\node[circle, draw, fill=black!0, inner sep=2pt, minimum width=2pt] (X3) at (-10,-16) {\footnotesize$3$};
%  	\node[circle, draw, fill=black!0, inner sep=2pt, minimum width=2pt] (X4) at (-8,-18) {\footnotesize$4$};
% \end{tikzpicture}

    \caption{Staged tree representations of $\GG_{1}=([4], \{1\to 3, 2\to 3, 3\to 4, 2\to 4\})$ and
$\GG_2=([4],\{1\to 2, 2\to 3 , 2\to 4, 3\to 4\})$ with respect to the linear extension $\pi=1234$ with binary random variables
.}
    \label{fig:staged tree representation}
\end{figure}

Let $\TT_\mathbb{G}$ denote the collection of all staged trees $\TT_\GG^\pi$ associated to any DAG and any of its linear extensions. 
The following proposition characterizes the elements of  $\TT_\mathbb{G}$ in terms of the stages of the trees.
% Let $\TT_{\mathbb{G}}$ denote the collection of all staged trees $(\TT,\theta)$ satisfying $\MM_{(\TT,\theta)} = \MM(\GG)$ for some DAG $\GG$.  
% The following proposition characterizes those compatibly labeled staged trees that represent DAG models; i.e, those compatbily labeled staged trees in $\TT_{\mathbb{G}}$.
%---PROPOSITION: Characterization of Staged Trees for DAGs----
\begin{proposition}
\label{prop: characterization DAG staged trees}
A staged tree $\TT$ is in $\TT_{\mathbb{G}}$ if and only it is compatibly labeled with levels $(L_1,\ldots,L_p)\sim(X_{\pi_1},\ldots,X_{\pi_p})$ for some $\pi\in\mathfrak{S}_p$, and if for all $k\in[p-1]$ the level $L_k = \RR_{\{\pi_1,\ldots,\pi_k\}}$ is partitioned into stages 
\[
\bigsqcup_{\yy\in\RR_{\Pi_k}}S_{\yy}
\]
for some subset $\Pi_{k+1}\subset\{\pi_1,\ldots,\pi_k\}$, where 
$
S_{\yy} = \{ \xx\in L_k: \xx_{\Pi_{k+1}} = \yy\}.
$
\end{proposition}

\begin{proof}
% To see the ``if'' direction ($\Rightarrow$), we need to show that if $(\TT,\theta)$ is a compatibly labeled staged tree with levels $(L_1,\ldots,L_p)\sim(X_{\pi_1},\ldots,X_{\pi_p})$ for some $\pi\in\mathfrak{S}_p$ for which $\MM_{(\TT,\theta)} = \MM(\GG)$ for some DAG $\GG$ then $\TT$ satisfies the specified properties.  
% Since $\MM_{(\TT,\theta)} = \MM(\GG)$, it follows that any $f\in\MM(\GG)$ factorizes according to $\TT$. 
% Since $\TT$ is compatibly labeled, it is uniform and stratified.  
% Hence, for all edges $e = x_{\pi_1}\cdots x_{\pi_k}\rightarrow x_{\pi_1}\cdots x_{\pi_k}x_{\pi_{k+1}}$ of $\TT$, we have
% \[
% x_{\theta(e)} = f(x_{\pi_{k+1}} \mid x_{\pi_1}\cdots x_{\pi_k}).
% \]
% Since $\MM_{(\TT,\theta)} = \MM(\GG)$, it follows that $f$ must also satisfy the factorization given in equation~\eqref{eqn: DAG factorization}, or equivalently, $f$ satisfies
% \[
% f(x_{\pi_{k+1}} \mid x_{\pa_\GG(k+1)}, x_{\{{\pi_1},\ldots, {\pi_k}\}\setminus\pa_\GG(k+1)}) = f(x_{\pi_{k+1}} \mid x_{\pa_\GG(k+1)}, x^\prime_{\{{\pi_1},\ldots, {\pi_k}\}\setminus\pa_\GG(k+1)})
% \]
% for all $x_{\pi_{k+1}}\in\RR_{\{\pi_{k+1}\}}$, $x_{\pa_\GG(k+1)}\in\RR_{\pa_\GG(k+1)}$ and 
% \[x_{\{{\pi_1},\ldots, {\pi_k}\}\setminus\pa_\GG(k+1)}, x^\prime_{\{{\pi_1},\ldots, {\pi_k}\}\setminus\pa_\GG(k+1)}\in\RR_{\{{\pi_1},\ldots, {\pi_k}\}\setminus\pa_\GG(k+1)}.
% \]

The ``only if" direction ($\Rightarrow$) follows from the construction of $\TT_\mathbb{G}$.  
So it only remains to show that any compatibly labeled staged tree $\TT$ with levels $(L_1,\ldots,L_p)\sim(X_{\pi_1},\ldots,X_{\pi_p})$ and the specified stages $S_{\yy}$ in each level $L_k$ is a staged tree associated to some DAG $\GG$ and one of its linear extensions.  
However, this follows in a straightforward way by taking $\GG$ to be the DAG $\GG = ([p],E)$ where $\pa_\GG(k+1) := \Pi_{k+1}$ for all $k\in[p-1]$ and $\pa_\GG(\pi_1) := \emptyset$. 
It is immediate that $\pi$ is a linear extension of $\GG$ and that the staging of $\TT$ coincides with the staging of $\TT_\GG^\pi$. 
\end{proof}

\begin{remark}
Note that, by definition, if $\TT\in \TT_{\mathbb{G}}$ then 
$\MM_{(\TT,\theta)}=\MM(\GG)$ for some DAG $\GG$.
It is likely that the converse holds, in particular if $\TT$
is compatibly labeled and $\MM_{(\TT,\theta)}=\MM(\GG)$ then $\TT\in \TT_{\mathbb{G}}$.
% While Proposition~\ref{prop: characterization DAG staged trees} is sufficient to derive a characterization of decomposable models in terms of their staged tree representations (as we will do in the coming section), it is also likely that an even stronger statement holds.  
% In particular, it would be interesting to see if $\TT_\mathbb{G}$ is exactly the set of all compatibly labeled staged trees $(\TT,\theta)$ such that $\MM_{(\TT,\theta)} = \MM(\GG)$ for some DAG $\GG$. 
\end{remark}

%---SECTION: Balanced Models------------
\section{Balanced Models}
\label{sec: balanced}

The family of balanced staged tree models was introduced in \cite{DG20} to characterize those staged tree models for which a certain pair of associated polynomial ideals coincide (see \cite[Theorem 3.1]{DG20}).  
This observation suggests that the balanced staged tree models are a generalization of decomposable models to more general staged trees.  
In this section, we prove that this is indeed the case (see Theorem~\ref{thm: perfectly balanced and stratified}).

For a staged tree $\TT = (V,E)$ and a node $v\in V$, we let $\TT_v$ denote the rooted subtree of $\TT$ whose root node is $v$. 
If we let $\Lambda_v$ denote the set of root-to-leaf paths in $\TT_v$, then the \emph{interpolating polynomial of $\TT_v$} is 
% Let $\TT = (V,E)$ be a staged tree. 
% For a fixed node $v\in V$, $\TT_v$ denotes the rooted subtree of $\TT$ rooted at node $v$.  
% Then $\TT_v$ is a staged tree with labeling induced by $\theta$.  
% We  denote the set of root-to-leaf paths in $\TT_{v}$ by $\Lambda_v$. The \emph{interpolating 
% polynomial of $\TT_v$} is
\[
t(v) := \sum_{\lambda\in\Lambda_v}\prod_{e\in \lambda}\theta(e)
\]
% where $\lambda$ is the set
% of edges in $\lambda$. 
 The polynomial $t(v)$ is an element of the polynomial ring $\R[\Theta_{\TT}]:=\R[\theta(e): e\in E]$ with one indeterminate for each edge label in $\LL$.
When $v$ is the root of $\TT$,  $t(v)$ is called the \emph{interpolating polynomial of $\TT$} \cite{GS18}. 
% Interpolating polynomials are useful to capture symmetries of subtrees of $(\TT,\theta)$.
%---DEFINITION: Balanced---
\begin{definition}
\label{def: balanced}
Let $(\TT,\theta)$ be a staged tree and $v,w\in V$ be two vertices in the same stage with children $\ch_\TT(v) = \{v_0,\ldots, v_k\}$ and $\ch_\TT(w) :=\{w_0,\ldots, w_k\}$, respectively.  
After a possible reindexing, we may assume that $\theta(v\rightarrow v_i) = \theta(w\rightarrow w_i)$ for all $i\in [k]_0$.  
The pair of vertices $v, w$ is \emph{balanced} if
\[
t(v_i)t(w_j) = t(w_i)t(v_j) \text{ in $\R[\Theta_\TT]$ for all $i\neq j\in[k]_0$.}
\]
The staged tree $(\TT,\theta)$ is called \emph{balanced} if every pair of vertices in the same stage is balanced.  
\end{definition}

\begin{example}
Two vertices $v,w$ in a staged tree $(\TT,\theta)$ are in the same position if
$t(v)=t(w)$. The staged tree $(\TT,\theta)$ is \emph{simple} if every pair of vertices
in the same stage is also in the same position. A simple and compatibly labeled
staged tree $(\TT,\theta)$ with levels $(L_1,\ldots,L_k)=(X_1,\ldots, X_p)$
is always balanced (see \cite[Lemma 2.12]{AD19}). The staged tree $\TT_{\GG_2}$ in
Figure~\ref{fig:staged tree representation} is simple and therefore balanced. The Theorem~\ref{thm: perfectly balanced and stratified},
states that the staged trees that represent decomposable models
are balanced. Hoever, balanced and compatibly labeled staged trees
are a larger class.
The two staged trees in Figure~\ref{fig:balanced} are balanced, compatibly labeled, and do not represent a
DAG model. 
\end{example}

In \cite{DG20}, the authors suggested balanced staged tree models as the natural generalization of perfect DAG models based on the fact that balanced staged tree models exhibit similar algebraic properties to perfect DAG models (see Section~\ref{sec: application} for more details).  
The following theorem shows that balanced staged tree models do indeed generalize perfect DAG models. 

%---THEOREM: Perfectly Balanced and Stratified----
\begin{theorem}
\label{thm: perfectly balanced and stratified}
Let $\GG = ([p],E)$ be a DAG, and let $\pi$ be a linear extension of $\GG$.
Then the following are equivalent:
\begin{enumerate}
    \item $\MM(\GG)$ is decomposable,
    \item $\GG$ is a perfect DAG, and
    \item the staged tree $\TT_\GG^\pi$ is balanced.
\end{enumerate}
%--- Version 04.05.2021
%The staged tree $\TT_\GG^\pi$ is balanced if and only if $\MM(\GG)$ is decomposable.
%----- Previous Statement of the theorem
% Let $\GG = ([p],E)$ be a DAG. Then the following are equivalent: 
% \begin{enumerate}
% \item The staged tree $\TT_\GG$ is balanced, 
% \item $\GG$ is a perfect DAG, and 
% \item $\ker(\Psi_{\TT_{\GG}}^{\toric})=\ker(\Psi_{\TT_{\GG}})=\ker(\Phi_{\GG})$.   
% \end{enumerate}
%(See Proposition~\ref{thmB: perfectly balanced and stratified} in Appendix~\ref{appendix:proofs}).
\end{theorem}

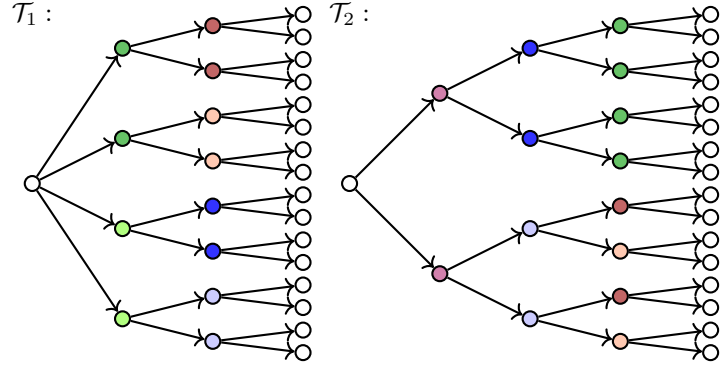
\begin{figure}
    \centering
     \begin{tikzpicture}[thick,scale=0.3]
	
	%---NODES---	 	
 	 \node[circle, draw, fill=black!0, inner sep=2pt, minimum width=1pt] (w3) at (0,0)  {};

 	 \node[circle, draw, fill=black!0, inner sep=2pt, minimum width=1pt] (w4) at (0,-1) {};
 	 \node[circle, draw, fill=black!0, inner sep=2pt, minimum width=1pt] (w5) at (0,-2) {};
 	 \node[circle, draw, fill=black!0, inner sep=2pt, minimum width=1pt] (w6) at (0,-3) {};
	 
 	 \node[circle, draw, fill=black!0, inner sep=2pt, minimum width=1pt] (v3) at (0,-4)  {};
 	 \node[circle, draw, fill=black!0, inner sep=2pt, minimum width=1pt] (v4) at (0,-5) {};
 	 \node[circle, draw, fill=black!0, inner sep=2pt, minimum width=1pt] (v5) at (0,-6) {};
 	 \node[circle, draw, fill=black!0, inner sep=2pt, minimum width=1pt] (v6) at (0,-7) {};

	 \node[circle, draw, fill=red!60!black!60, inner sep=2pt, minimum width=2pt] (w1) at (-4,-.5) {};
 	 \node[circle, draw, fill=red!60!black!60, inner sep=2pt, minimum width=2pt] (w2) at (-4,-2.5) {}; 

 	 \node[circle, draw, fill=red!70!yellow!30, inner sep=2pt, minimum width=2pt] (v1) at (-4,-4.5) {};
 	 \node[circle, draw, fill=red!70!yellow!30, inner sep=2pt, minimum width=2pt] (v2) at (-4,-6.5) {};

 	 \node[circle, draw, fill=green!60!black!60, inner sep=2pt, minimum width=2pt] (w) at (-8,-1.5) {};

 	 \node[circle, draw, fill=green!60!black!60, inner sep=2pt, minimum width=2pt] (v) at (-8,-5.5) {};

 	 %\node[circle, draw, fill=violet!75, inner sep=2pt, minimum width=2pt] (r) at (-12,-3.5) {};

 	 \node[circle, draw, fill=black!0, inner sep=2pt, minimum width=2pt] (w3i) at (0,-8)  {};
 	 \node[circle, draw, fill=black!0, inner sep=2pt, minimum width=2pt] (w4i) at (0,-9) {};
 	 \node[circle, draw, fill=black!0, inner sep=2pt, minimum width=2pt] (w5i) at (0,-10) {};
 	 \node[circle, draw, fill=black!0, inner sep=2pt, minimum width=2pt] (w6i) at (0,-11) {};
	 
 	 \node[circle, draw, fill=black!0, inner sep=2pt, minimum width=2pt] (v3i) at (0,-12)  {};
 	 \node[circle, draw, fill=black!0, inner sep=2pt, minimum width=2pt] (v4i) at (0,-13) {};
 	 \node[circle, draw, fill=black!0, inner sep=2pt, minimum width=2pt] (v5i) at (0,-14) {};
 	 \node[circle, draw, fill=black!0, inner sep=2pt, minimum width=2pt] (v6i) at (0,-15) {};

	 \node[circle, draw, fill=blue!80, inner sep=2pt, minimum width=2pt] (w1i) at (-4,-8.5) {};
 	 \node[circle, draw, fill=blue!80, inner sep=2pt, minimum width=2pt] (w2i) at (-4,-10.5) {};
 	 \node[circle, draw, fill= blue!20 , inner sep=2pt, minimum width=2pt] (v1i) at (-4,-12.5) {};
 	 \node[circle, draw, fill= blue!20, inner sep=2pt, minimum width=2pt] (v2i) at (-4,-14.5) {};

 	 \node[circle, draw, fill=green!60!yellow!50, inner sep=2pt, minimum width=2pt] (wi) at (-8,-9.5) {};

 	 \node[circle, draw, fill=green!60!yellow!50, inner sep=2pt, minimum width=2pt] (vi) at (-8,-13.5) {};

 	 \node[circle, draw, fill=violet!0, inner sep=2pt, minimum width=2pt] (ri) at (-12,-7.5) {};

 	 %\node[circle, draw, fill=black!0, inner sep=2pt, minimum width=2pt] (I) at (-16,-7.5) {};
 	 \node (T2) at (-12,0) {$\TT_{1}:$};
 	%-- DAG

	%---EDGES---	 
 	% \draw[->]   (I) --    (r) ;
 	%\draw[->]   (I) --   (ri) ;

 	 \draw[->]   (ri) --   (w) ;
 	 \draw[->]   (ri) --   (v) ;

 	 \draw[->]   (w) --  (w1) ;
 	 \draw[->]   (w) --  (w2) ;

 	 \draw[->]   (w1) --   (w3) ;
 	 \draw[->]   (w1) --   (w4) ;
 	 \draw[->]   (w2) --  (w5) ;
 	 \draw[->]   (w2) --  (w6) ;

 	 \draw[->]   (v) --  (v1) ;
 	 \draw[->]   (v) --  (v2) ;

 	 \draw[->]   (v1) --  (v3) ;
 	 \draw[->]   (v1) --  (v4) ;
 	 \draw[->]   (v2) --  (v5) ;
 	 \draw[->]   (v2) --  (v6) ;

 	 \draw[->]   (ri) --   (wi) ;
 	 \draw[->]   (ri) -- (vi) ;

 	 \draw[->]   (wi) --  (w1i) ;
 	 \draw[->]   (wi) --  (w2i) ;

 	 \draw[->]   (w1i) --  (w3i) ;
 	 \draw[->]   (w1i) -- (w4i) ;
 	 \draw[->]   (w2i) --  (w5i) ;
 	 \draw[->]   (w2i) --  (w6i) ;

 	 \draw[->]   (vi) --  (v1i) ;
 	 \draw[->]   (vi) --  (v2i) ;

 	 \draw[->]   (v1i) --  (v3i) ;
 	 \draw[->]   (v1i) -- (v4i) ;
 	 \draw[->]   (v2i) -- (v5i) ;
 	 \draw[->]   (v2i) --  (v6i) ;

\end{tikzpicture}
 \begin{tikzpicture}[thick,scale=0.3]

 	 \node[circle, draw, fill=black!0, inner sep=2pt, minimum width=1pt] (w3) at (0,0)  {};

 	 \node[circle, draw, fill=black!0, inner sep=2pt, minimum width=1pt] (w4) at (0,-1) {};
 	 \node[circle, draw, fill=black!0, inner sep=2pt, minimum width=1pt] (w5) at (0,-2) {};
 	 \node[circle, draw, fill=black!0, inner sep=2pt, minimum width=1pt] (w6) at (0,-3) {};
	 
 	 \node[circle, draw, fill=black!0, inner sep=2pt, minimum width=1pt] (v3) at (0,-4)  {};
 	 \node[circle, draw, fill=black!0, inner sep=2pt, minimum width=1pt] (v4) at (0,-5) {};
 	 \node[circle, draw, fill=black!0, inner sep=2pt, minimum width=1pt] (v5) at (0,-6) {};
 	 \node[circle, draw, fill=black!0, inner sep=2pt, minimum width=1pt] (v6) at (0,-7) {};

	 \node[circle, draw, fill=green!60!black!60, inner sep=2pt, minimum width=2pt] (w1) at (-4,-.5) {};
 	 \node[circle, draw, fill=green!60!black!60, inner sep=2pt, minimum width=2pt] (w2) at (-4,-2.5) {}; 

 	 \node[circle, draw, fill=green!60!black!60, inner sep=2pt, minimum width=2pt] (v1) at (-4,-4.5) {};
 	 \node[circle, draw, fill=green!60!black!60, inner sep=2pt, minimum width=2pt] (v2) at (-4,-6.5) {};

 	 \node[circle, draw, fill=blue!80, inner sep=2pt, minimum width=2pt] (w) at (-8,-1.5) {};

 	 \node[circle, draw, fill=blue!80, inner sep=2pt, minimum width=2pt] (v) at (-8,-5.5) {};

 	 \node[circle, draw, fill=violet!70!red!50, inner sep=2pt, minimum width=2pt] (r) at (-12,-3.5) {};

 	 \node[circle, draw, fill=black!0, inner sep=2pt, minimum width=2pt] (w3i) at (0,-8)  {};
 	 \node[circle, draw, fill=black!0, inner sep=2pt, minimum width=2pt] (w4i) at (0,-9) {};
 	 \node[circle, draw, fill=black!0, inner sep=2pt, minimum width=2pt] (w5i) at (0,-10) {};
 	 \node[circle, draw, fill=black!0, inner sep=2pt, minimum width=2pt] (w6i) at (0,-11) {};
	 
 	 \node[circle, draw, fill=black!0, inner sep=2pt, minimum width=2pt] (v3i) at (0,-12)  {};
 	 \node[circle, draw, fill=black!0, inner sep=2pt, minimum width=2pt] (v4i) at (0,-13) {};
 	 \node[circle, draw, fill=black!0, inner sep=2pt, minimum width=2pt] (v5i) at (0,-14) {};
 	 \node[circle, draw, fill=black!0, inner sep=2pt, minimum width=2pt] (v6i) at (0,-15) {};

	 \node[circle, draw, fill=red!60!black!60, inner sep=2pt, minimum width=2pt] (w1i) at (-4,-8.5) {};
 	 \node[circle, draw, fill=red!70!yellow!30, inner sep=2pt, minimum width=2pt] (w2i) at (-4,-10.5) {};

 	 \node[circle, draw, fill= red!60!black!60 , inner sep=2pt, minimum width=2pt] (v1i) at (-4,-12.5) {};
 	 \node[circle, draw, fill= red!70!yellow!30, inner sep=2pt, minimum width=2pt] (v2i) at (-4,-14.5) {};

 	 \node[circle, draw, fill=blue!20, inner sep=2pt, minimum width=2pt] (wi) at (-8,-9.5) {};

 	 \node[circle, draw, fill=blue!20, inner sep=2pt, minimum width=2pt] (vi) at (-8,-13.5) {};

 	 \node[circle, draw, fill=violet!70!red!50, inner sep=2pt, minimum width=2pt] (ri) at (-12,-11.5) {};

 	 \node[circle, draw, fill=black!0, inner sep=2pt, minimum width=2pt] (I) at (-16,-7.5) {};
 	 
 	  \node (T2) at (-16,0) {$\TT_{2}:$};

	%---EDGES---	 
 	 \draw[->]   (I) --    (r) ;
 	 \draw[->]   (I) --   (ri) ;

 	 \draw[->]   (r) --   (w) ;
 	 \draw[->]   (r) --   (v) ;

 	 \draw[->]   (w) --  (w1) ;
 	 \draw[->]   (w) --  (w2) ;

 	 \draw[->]   (w1) --   (w3) ;
 	 \draw[->]   (w1) --   (w4) ;
 	 \draw[->]   (w2) --  (w5) ;
 	 \draw[->]   (w2) --  (w6) ;

 	 \draw[->]   (v) --  (v1) ;
 	 \draw[->]   (v) --  (v2) ;

 	 \draw[->]   (v1) --  (v3) ;
 	 \draw[->]   (v1) --  (v4) ;
 	 \draw[->]   (v2) --  (v5) ;
 	 \draw[->]   (v2) --  (v6) ;

 	 \draw[->]   (ri) --   (wi) ;
 	 \draw[->]   (ri) -- (vi) ;

 	 \draw[->]   (wi) --  (w1i) ;
 	 \draw[->]   (wi) --  (w2i) ;

 	 \draw[->]   (w1i) --  (w3i) ;
 	 \draw[->]   (w1i) -- (w4i) ;
 	 \draw[->]   (w2i) --  (w5i) ;
 	 \draw[->]   (w2i) --  (w6i) ;

 	 \draw[->]   (vi) --  (v1i) ;
 	 \draw[->]   (vi) --  (v2i) ;

 	 \draw[->]   (v1i) --  (v3i) ;
 	 \draw[->]   (v1i) -- (v4i) ;
 	 \draw[->]   (v2i) -- (v5i) ;
 	 \draw[->]   (v2i) --  (v6i) ;

\end{tikzpicture}
%  \begin{tikzpicture}[thick,scale=0.3]
%  	%-- Nodes for the DAG
%  	\node[circle, draw, fill=black!0, inner sep=2pt, minimum width=2pt] (X1) at (-16,-18) {\footnotesize$1$};
%  	\node[circle, draw, fill=black!0, inner sep=2pt, minimum width=2pt] (X2) at (-12,-18) {\footnotesize$2$};
%  	\node[circle, draw, fill=black!0, inner sep=2pt, minimum width=2pt] (X3) at (-10,-16) {\footnotesize$3$};
%  	\node[circle, draw, fill=black!0, inner sep=2pt, minimum width=2pt] (X4) at (-8,-18) {\footnotesize$4$};
%  	\draw[->] (X1) -- (X2)
%  	\draw[->] (X2) -- (X4)
% \end{tikzpicture}
%  \begin{tikzpicture}[thick,scale=0.3]

%  	%-- Nodes for the DAG
%  	\node[circle, draw, fill=black!0, inner sep=2pt, minimum width=2pt] (X1) at (-16,-18) {\footnotesize$1$};
%  	\node[circle, draw, fill=black!0, inner sep=2pt, minimum width=2pt] (X2) at (-12,-18) {\footnotesize$2$};
%  	\node[circle, draw, fill=black!0, inner sep=2pt, minimum width=2pt] (X3) at (-10,-16) {\footnotesize$3$};
%  	\node[circle, draw, fill=black!0, inner sep=2pt, minimum width=2pt] (X4) at (-8,-18) {\footnotesize$4$};
% \end{tikzpicture}

    \caption{Balanced staged trees that are not simple and do not represent a DAG model.}
    \label{fig:balanced}
    \end{figure}
\subsection{Proof of Theorem~\ref{thm: perfectly balanced and stratified}}
Before we present the proof of the theorem we will prove a technical lemma:
%---LEMMA: Simplified Balanced Condition---
\begin{lemma}
\label{lem: simplified balanced}
Let $\GG = ([p],E)$ be a DAG and assume that $\pi = 12\cdots p$ is a linear extension of $\GG$. Then $\TT_\GG$ is balanced
if and only if for every pair of vertices $v,w$ in the same stage  with $v=x_1\cdots x_i, w=x_1'\cdots x_i'\in\RR_{[i]}$
%for outcomes $x_k,x_k'\in \RR_{\{k\}}$, $k\in [i]$, 
there exists a bijection
\begin{equation*}
\begin{split}
\Phi:& \RR_{[p]\setminus[i+1]}\times \RR_{[p]\setminus[i+1]}\longrightarrow\RR_{[p]\setminus[i+1]}\times \RR_{[p]\setminus[i+1]}, \\
    & (y_{i+2}\cdots y_p, y_{i+2}^\prime\cdots y_p^\prime) \mapsto (z_{i+2}\cdots z_p, z_{i+2}^\prime\cdots z_p^\prime)
\end{split}
\end{equation*}
such that for all $k\geq i+2$ and all $s\neq r \in [d_{i+1}]$
\begin{equation}
\label{eqnL: simplified balanced}
\begin{split}
f&(y_k\mid (x_1\cdots x_i,s,y_{i+2}\cdots y_p)_{\pa_\GG(k)})f(y_k^\prime\mid (x_1^\prime\cdots x_i^\prime,r,y_{i+2}^\prime\cdots y_p^\prime)_{\pa_\GG(k)}) \\
&= f(z_k\mid (x_1^\prime\cdots x_i^\prime,s,z_{i+2}\cdots z_p)_{\pa_\GG(k)})f(z_k^\prime\mid (x_1\cdots x_i,r,z_{i+2}^\prime\cdots z_p^\prime)_{\pa_\GG(k)}).
\end{split}
\end{equation}
\end{lemma}
\begin{proof}
Suppose $\TT_\GG$ is balanced and let $v,w$ be two vertices in the same stage as in the statement of this lemma.
Since the pair $v,w$ is balanced, 
\begin{equation}
\label{eqnL: balanced}
t(x_1\cdots x_is)t(x_1^\prime\cdots x_i^\prime r) = t(x_1\cdots x_ir)t(x_1^\prime\cdots x_i^\prime s)
\end{equation}
for all $s\neq r \in [d_{i+1}]$. Next we rewrite the two factors on each side of the equality (\ref{eqnL: balanced}) 
using the definition of $t(\cdot)$ and multiply out the two expressions on each side. For the leftmost factor
we have
\[
t(x_1\cdots x_i s) = \sum_{x_{i+2}\cdots x_p\in\RR_{[p]\setminus[i+1]}}\prod_{k = i+2}^pf(x_k\mid (x_1\cdots x_i,s,x_{i+2}\cdots x_p)_{\pa_\GG(k)}).
\] 
Using a similar expression for the other three factors, multiplying out and writing as a double summation, the equation~\eqref{eqnL: balanced} becomes
\begin{equation}
\label{eqnE: balanced}
\begin{split}
\sum_{\substack{x_{i+2}\cdots x_p \in \RR_{[p]\setminus[i+1]} \\
     x_{i+2}'\cdots x_p' \in \RR_{[p]\setminus[i+1]}} } \!\!\!\!\!\!\!\!\!\!\!\!\!\!\! \prod f(x_{k}|(x_1\cdots x_i,s,x_{i+2}\cdots
     x_p)_{\pa_{\GG}(k)}) f(x_{k}'|(x_1'\cdots x_i',r,x_{i+2}'\cdots x_p')_{\pa_{\GG}(k)}) \\
 = \sum_{\substack{\overline{x_{i+2}\cdots x_p }\in \RR_{[p]\setminus[i+1]} \\
     \overline{x_{i+2}'\cdots x_p'} \in \RR_{[p]\setminus[i+1]}} } \!\!\!\!\!\!\!\!\!\!\!\!\!\!\! \prod 
     f(x_{k}|(x_1\cdots x_i,r,\overline{x_{i+2}\cdots x_p})_{\pa_{\GG}(k)}) f(x_{k}'|(x_1'\cdots x_i',s,\overline{x_{i+2}'\cdots x_p'})_{\pa_{\GG}(k)})
     \end{split}.
\end{equation}
Where the product inside both summations in (\ref{eqnE: balanced}) is taken from $k=i+2$ to $k=p$.
The expression (\ref{eqnE: balanced}) is an equality in the polynomial ring where the labels in $\LL$ are treated as indeterminates. % $\R[\Theta_\LL]$ \textcolor{red}{need to define this ring...}. 
Since each of the terms in the sum is a monomial of degree $2(p-i-1)$, (\ref{eqnE: balanced}) is an equality of homogenous
polynomials. Moreover, each side of (\ref{eqnE: balanced}) has the same number of terms and each term 
has coefficient equal to one. Hence there exists a bijection between terms in the left-hand-side of 
(\ref{eqnE: balanced}) and terms in its right-hand-side. 
We denote this bijection by $\Phi:\RR_{[p]\setminus[i+1]}\times \RR_{[p]\setminus[i+1]}\longrightarrow\RR_{[p]\setminus[i+1]}\times \RR_{[p]\setminus[i+1]}$ 
as in the statement of the lemma. 
Under this bijection it is true that
\begin{equation}
\label{eqnS: balanced}
\begin{split}
\prod_{k=i+2}^{p}f(y_{k}|(x_1\cdots x_i,s,y_{i+2}\cdots
     y_p)_{\pa_{\GG}(k)}) f(y_{k}'|(x_1'\cdots x_i',r,y_{i+2}'\cdots y_p')_{\pa_{\GG}(k)}) 
 = \\ \prod_{k=i+2}^{p}
     f(z_{k}|(x_1\cdots x_i,r,z_{i+2}\cdots z_p)_{\pa_{\GG}(k)}) f(z_{k}'|(x_1'\cdots x_i',s,z_{i+2}'\cdots 
     z_p')_{\pa_{\GG}(k)}).
     \end{split}
\end{equation}
Since $\TT_\GG$ is stratified, any two vertices in the same stage must be in the same level. Thus 
from (\ref{eqnS: balanced}) we obtain the desired equality
\begin{equation}
\label{eqnS: Sbalanced}
\begin{split}
f(y_{k}|(x_1\cdots x_i,s,y_{i+2}\cdots
     y_p)_{\pa_{\GG}(k)}) f(y_{k}'|(x_1'\cdots x_i',r,y_{i+2}'\cdots y_p')_{\pa_{\GG}(k)}) 
 = \\ 
     f(z_{k}|(x_1\cdots x_i,r,z_{i+2}\cdots z_p)_{\pa_{\GG}(k)}) f(z_{k}'|(x_1'\cdots x_i',s,z_{i+2}'\cdots 
     z_p')_{\pa_{\GG}(k)}).
     \end{split}
\end{equation}
To check the other direction, it sufficient to note that we can trace backwards the steps in the proof to conclude
that the pair $v,w$ is balanced provided there exists a bijection $\Phi$ that satisfies (\ref{eqnS: Sbalanced}).
\end{proof}

We can then prove Theorem~\ref{thm: perfectly balanced and stratified}. 

\subsubsection{Proof of Theorem~\ref{thm: perfectly balanced and stratified}}
\label{subsec: proof}
% The equivalence of $(1)$ and $(3)$ is established by \cite[Theorem 10]{DG20}. 
% A proof of the equivalence of $(1)$ and $(2)$ is in Appendix~\ref{app:proofs} (see Proposition~\ref{thmB: perfectly balanced and stratified}). 
% We now prove the equivalence of $(1)$ and $(2)$.
% It relies in part on a technical lemma given in Appendix~\ref{app:proofs} (see  Lemma~\ref{lem: simplified balanced}).
% All necessary graph theory terms are defined in Appendix~\ref{app:definitions}.
For all necessary graph theory terminology, we refer the reader to \cite{L96}.
The equivalence $(1) \Leftrightarrow (2)$ is \cite[Proposition 3.28]{L96}.
$(3)\Leftarrow(2)$: Let $\GG$ be a perfect DAG with linear extension $\pi = 12\cdots p$.  
Then $(1,2,\ldots, p)$ is a perfect elimination ordering for the skeleton of $\GG$. 
Suppose that $v,w\in V$ are in the same stage.  
% By Proposition~\ref{prop: DAGs are stratified}, $\TT_\GG$ is stratified. 
Since $\TT_\GG$ is stratified, $v$ and $w$ are in the same level, say level $i$.
Therefore, $v = x_1\cdots x_i$ and $w = x_1^\prime\cdots x_i^\prime$ for some outcomes $x_k,x_k^\prime\in\RR_{\{k\}}$ for $k\in[i]$.  
Using the characterization in Lemma~\ref{lem: simplified balanced} we must find a bijection
\begin{equation*}
\begin{split}
\Phi:& \RR_{[p]\setminus[i+1]}\times \RR_{[p]\setminus[i+1]}\longrightarrow\RR_{[p]\setminus[i+1]}\times \RR_{[p]\setminus[i+1]}, \\
    & (y_{i+2}\cdots y_p, y_{i+2}^\prime\cdots y_p^\prime) \mapsto (z_{i+2}\cdots z_p, z_{i+2}^\prime\cdots z_p^\prime)
\end{split}
\end{equation*}
such that for all $k\geq i+2$
\begin{equation}
\label{eqn: simplified balanced}
\begin{split}
f&(y_k\mid (x_1\cdots x_i,s,y_{i+2}\cdots y_p)_{\pa_\GG(k)})f(y_k^\prime\mid (x_1^\prime\cdots x_i^\prime,r,y_{i+2}^\prime\cdots y_p^\prime)_{\pa_\GG(k)}) \\
&= f(z_k\mid (x_1^\prime\cdots x_i^\prime,s,z_{i+2}\cdots z_p)_{\pa_\GG(k)})f(z_k^\prime\mid (x_1\cdots x_i,r,z_{i+2}^\prime\cdots z_p^\prime)_{\pa_\GG(k)})
\end{split}
\end{equation}
whenever $s\neq r\in[d_{i+1}]$. 
To this end, we define the bijection $\Phi$ in the following way:  
Given a pair of outcomes $(y_{i+2}\cdots y_p, y_{i+2}^\prime\cdots y_p^\prime) \in \RR_{[p]\setminus[i+1]}\times \RR_{[p]\setminus[i+1]}$ define the pair $(z_{i+2}\cdots z_p, z_{i+2}^\prime\cdots z_p^\prime)$ by the rule:
For $k\geq i+2$,
\begin{itemize}
	\item if $i+1$ is an ancestor of $k$ then set $z_k:= y_k$ and $z_k^\prime := y_k^\prime$, and
	\item if $i+1$ is not an ancestor of $k$ then set $z_k:= y_k^\prime$ and $z_k^\prime := y_k$.
\end{itemize}
Notice that $\Phi$ is a bijection since it is an involution.
To prove that \eqref{eqn: simplified balanced} is satisfied with respect to the chosen bijection $\Phi$, we must check that it holds for all $k\geq i+2$ whenever $s\neq r\in[d_{i+1}]$.  
In the following, suppose that $s\neq r\in[d_{i+1}]$, and let $k\geq i+2$.  
It follows that $i+1$ is either an ancestor of $k$ or it is not.  
We will show that \eqref{eqn: simplified balanced} holds in both of these two cases, which will complete the proof.  

In the first case, suppose that $i+1\in\an_\GG(k)$. 
To show that \eqref{eqn: simplified balanced} holds in this case, 
    it suffices to show that 
    $(x_1\cdots x_i,s,y_{i+2}\cdots y_p)_{\pa_\GG(k)} = (x_1^\prime\cdots x_i^\prime,s,z_{i+2}\cdots z_{i+p})_{\pa_\GG(k)}$ and
    $(x_1\cdots x_i,s,y_{i+2}^\prime\cdots y_p^\prime)_{\pa_\GG(k)} = (x_1^\prime\cdots x_i^\prime,s,z_{i+2}^\prime\cdots 
    z_{i+p}
    ^\prime)_{\pa_\GG(k)}$. Breaking these two equalities into equalities of subsequences, it suffices
    to prove $(x_1\cdots x_i)_{\pa_\GG(k)}=(x_1^\prime \cdots x_i^\prime)_{\pa_\GG(k)}$, $(y_{i+2}\cdots y_p)_{\pa_\GG(k)}=
    (z_{i+2}\cdots z_p)_{\pa_\GG(k)}$ and $(y_{i+2}^\prime\cdots y_{p}^\prime)_{\pa_\GG(k)}=(z_{i+2}^\prime\cdots z_p^\prime)
    _{\pa_\GG(k)}$.
 
To see that $(x_1\cdots x_i)_{\pa_\GG(k)}=(x_1^\prime \cdots x_i^\prime)_{\pa_\GG(k)}$ holds, we first note that for $k=i+1$, since $x_1\cdots x_i$ and $x_1^\prime\cdots x_i^\prime$ are in the same stage then
\[
f(x_{i+1}\mid (x_1\cdots x_i)_{\pa_\GG(i+1)}) = f(x_{i+1}\mid (x_1^\prime\cdots x_i^\prime)_{\pa_\GG(i+1)})
\]
for all $x_{i+1}\in[d_{i+1}]$, and so $(x_1\cdots x_i)_{\pa_\GG(i+1)} = (x_1^\prime\cdots x_i^\prime)_{\pa_\GG(i+1)}$.
Thus, to show that $(x_1\cdots x_i)_{\pa_\GG(k)} = (x_1^\prime \cdots x_i^\prime)_{\pa_\GG(k)}$ for $k\geq i+2$, it suffices to show that $\pa_\GG(k)\cap[i]\subset\pa_\GG(i+1)$.  
To this end, suppose that $j\in\pa_\GG(k)$ and that $j<i+1$.  
Since $\pi = 12 \cdots p$ is a linear extension of $\GG$, we know that any descendant $\ell\in[p]$ of $i+1$ (including $k$) satisfies 
$\ell>i+1$.  
Therefore, since $i+1\in\an_\GG(k)$, $j\in\pa_\GG(k)$ and $j< i+1$, we know that there exists $k^\prime\in\pa_\GG(k)$ with $k^\prime \neq j$ satisfying $i+1\geq k'>k$.  
It then follows by the chordality of $\tilde{\GG}$ (the skeleton of $\GG$), and the fact that $\pi = 12\cdots p$ is a linear extension of $\GG$, that $j\in\pa_\GG(k^\prime)$.  
If $k'=i+1$ we are done, otherwise iterating this argument shows that $j\in\pa_\GG(i+1)$.  
Thus, we conclude that $(x_1\cdots x_i)_{\pa_\GG(k)} = (x_1^\prime\cdots x_i^\prime)_{\pa_\GG(k)}$.

To see that 
$(y_{i+2}\cdots y_p)_{\pa_\GG(k)} = (z_{i+2}\cdots z_p)_{\pa_\GG(k)}$, it suffices to show the slightly stronger statement that 
\[
(y_{i+2}\cdots y_p)_{\an_\GG(k)\cap[p]\setminus[i+1]} = (z_{i+2}\cdots z_p)_{\an_\GG(k)\cap[p]\setminus[i+1]}.
\]
By construction of the bijection $\Phi$, we know that 
\[
(y_{i+2}\cdots y_p)_{\de_\GG(i+1)} = (z_{i+2}\cdots z_p)_{\de_\GG(i+1)}.
\]
Hence, to prove the desired statement, it suffices to show that $\an_\GG(k)\cap[p]\setminus[i+1]\subset \de_\GG(i+1)$. 
To see this, suppose that $k^{\prime\prime}\in \an_\GG(k)\cap[p]\setminus[i+1]$.  
Let $[i+1,k]$ denote all nodes in $\GG$ that lie on a directed path from $i+1$ to $k$.  
If $k^{\prime\prime}\in[i+1,k]$, then $k^{\prime\prime}\in\de_\GG(i+1)$.  
So suppose $k^{\prime\prime}\not\in[i+1,k]$.  
Then there exists $k^{\prime\prime\prime}\in\de_\GG(k^{\prime\prime})$ such that $k^{\prime\prime\prime}\in[i+1,k]$.  
(For instance, $k$ is one such node.)  
Pick such a $k^{\prime\prime\prime}$ so that, over all such choices, the minimal length directed path from $k^{\prime\prime}$ to an element of $[i+1,k]$ is of shortest length.  
Further pick $k^{\prime\prime\prime}$ such that, over all such choices satisfying the previous condition, $k^{\prime\prime\prime}$ has minimum value in the natural order on $[p]$. 

Since $k^{\prime\prime\prime}\in\de_\GG(k^{\prime\prime})$ then there exists a directed path in $\GG$
\[
k^{\prime\prime} \rightarrow a_0 \rightarrow a_1 \rightarrow \cdots \rightarrow a_m \rightarrow k^{\prime\prime\prime}.
\]
We assume that this path is the shortest possible path from $k^{\prime\prime}$ to $[i+1,k]$, based on our previous assumptions.  
Since $k^{\prime\prime\prime}\in[i+1,k]$, there also exists a directed path in $\GG$
\[
i+1 \rightarrow b_0 \rightarrow b_1 \rightarrow \cdots \rightarrow b_t \rightarrow k^{\prime\prime\prime}.
\]
Since $b_t, a_m\in\pa_\GG(k^{\prime\prime\prime})$ and $\GG$ is perfect, we know that $b_t$ and $a_m$ are adjacent in $\GG$.  
Since our path from $k^{\prime\prime}$ to $[i+1,k]$ was to chosen to have shortest possible length, and since $k^{\prime\prime\prime}$ has minimum value over all such paths, we know that $b_t <a_m$. 
Thus, $b_t\rightarrow a_m$ is an edge of $\GG$.  
Since $b_t,a_{m-1}\in\pa_\GG(a_m)$, and since $\GG$ is perfect, we know that $b_t\rightarrow a_{m-1}$ is also an edge of $\GG$.  
Otherwise, we did not pick the shortest path to $[i+1,k]$.  
Iterating this argument shows that $b_t\rightarrow k^{\prime\prime}$ is an edge of $\GG$.  
This implies that $k^{\prime\prime}\in[i+1,k]$, which is a contradiction.  

Hence, we conclude that $\an_\GG(k)\cap[p]\setminus[i+1]\subset\de_\GG(i+1)$, and $(y_{i+2}\cdots y_p)_{\pa_\GG(k)} = (z_{i+2}\cdots z_p)_{\pa_\GG(k)}$. The same argument proves $(y_{i+2}^\prime\cdots y_{p}^\prime)_{\pa_\GG(k)}=(z_{i+2}^\prime\cdots z_p^\prime)
    _{\pa_\GG(k)}$.

It now only remains to check that \eqref{eqn: simplified balanced} holds whenever $i+1\notin\an_\GG(k)$. 
Since $i+1\notin\an_\GG(k)$, it follows that $i+1\notin\pa_\GG(k)$.  
So to prove the desired equality, it suffices to show that $(y_{i+2}\cdots y_p)_{\pa_\GG(k)} = (z_{i+2}^\prime\cdots z_p^\prime)_{\pa_\GG(k)}$ and $(y_{i+2}^\prime \cdots y_p^\prime)_{\pa_\GG(k)} = (z_{i+2}\cdots z_p)_{\pa_\GG(k)}$. 
However, since $i+1\notin\an_\GG(k)$, it also follows that no $j\in\pa_\GG(k)$ is in $\de_\GG(i+1)$.  
Hence, by the definition of $\Phi$, for all $j\in\pa_\GG(k)\cap[p]\setminus[i+1]$ we have that $z_j = y_j^\prime$ and $z_j^\prime = y_j$, which completes the first direction of the proof.

$(3)\Rightarrow(2)$: We prove the contrapositive, i.e. if $\GG$ is not perfect then $\TT_{\GG}$ is not balanced. Without loss of generality,
we assume the nodes of $\GG$ have a topological order. 
That is, if $u\to v$ is an arrow in $\GG$ then $u$ is less than $v$. If $\GG$ is not
a perfect DAG then $\GG$ has a collider $i\to l \leftarrow j$. 
By assumption $i<l$ and $j<l$, and we further assume $i<j$.
Since $i\notin\pa_{\GG}(j)$, there exist two outcomes $x_1\cdots x_p,x_1'\cdots x_p' \in \RR$ such that $x_i\neq x_i'$ and 
$(x_1\cdots x_p)_{\pa_{\GG}(j)}=(x_1'\cdots x_p')_{\pa_{\GG}(j)}$. The latter equality implies the vertices
$x_1\cdots x_{j-1}$ and $x_1'\cdots x_{j-1}'$ in $\TT_{\GG}$ are in the same stage. We show that the balanced condition cannot possibly hold for 
these two vertices. 

As in the proof of $(3)\Leftarrow(2)$, using Lemma~\ref{lem: simplified balanced}, the balanced condition for the vertices 
$x_1\cdots x_{j-1}$ and $x_1'\cdots x_{j-1}'$ holds if and only if there exist a bijection 
$\Phi:\RR_{[p]\setminus[j]}\times \RR_{[p]\setminus[j]}\longrightarrow\RR_{[p]\setminus[j]}\times \RR_{[p]\setminus[j]}$  such that for all $k\geq j+1$ and all $s,r\in[d_j]$
\begin{equation}
\label{eqn: simplified balanced 2}
\begin{split}
&f(y_k\mid (x_1\cdots x_{j-1},s,y_{j+1}\cdots y_p)_{\pa_\GG(k)})f(y_k^\prime\mid (x_1^\prime\cdots x_{j-1}^\prime,r,y_{j+1}^\prime\cdots y_p^\prime)_{\pa_\GG(k)}) \\
&= f(z_k\mid (x_1^\prime\cdots x_{j-1}^\prime,s,z_{j+1}\cdots z_p)_{\pa_\GG(k)})f(z_k^\prime\mid (x_1\cdots x_{j-1},r,z_{j+1}^\prime\cdots z_p^\prime)_{\pa_\GG(k)}).
\end{split}
\end{equation}
Thus to show $\TT_{\GG}$ is not balanced, we show that (\ref{eqn: simplified balanced 2}) cannot hold for $k=l$.
The only two ways to satisfy (\ref{eqn: simplified balanced 2}) are if 
\begin{align*}
f(y_l\mid (x_1\cdots x_{j-1},s,y_{j+1}\cdots y_p)_{\pa_\GG(k)}) &= f(z_l\mid (x_1^\prime\cdots x_{j-1}^\prime,s,z_{j+1}\cdots z_p)_{\pa_\GG(k)}),
\end{align*}
with $y_{l}=z_{l}$ or 
\begin{align*}
f(y_l\mid (x_1\cdots x_{j-1},s,y_{i+2}\cdots y_p)_{\pa_\GG(k)}) &= f(z_k^\prime\mid (x_1\cdots x_{j-1},r,z_{i+2}^\prime\cdots z_p^\prime)_{\pa_\GG(k)}),
\end{align*}
with $y_{l}=z_{l}'$.
The first equation cannot hold for any choice of $z_{j+1}\cdots z_{p}\in \RR_{[p]\setminus [j]}$ because 
$i\in\pa_{\GG}(l)$ and by construction $x_i\neq x_i'$ hence $(x_1\cdots x_{j-1}sy_{j+1}\cdots y_p)_{\pa_{\GG}(l)}
\neq (x_1'\cdots x_{j-1}'sz_{j+1}\cdots z_p)_{\pa_{\GG}(l)}$. The second equation cannot hold for any choice of 
 $z_{j+1}'\cdots z_{p}'\in \RR_{[p]\setminus [j]}$ because $j\in \pa_{\GG}(l)$ and $s\neq r$ so 
 $(x_1\cdots x_{j-1}sy_{j+1}\cdots y_p)_{\pa_{\GG}(l)}\neq (x_1\cdots x_{j-1}rz_{j+1}'\cdots z_p')_{\pa_{\GG}(l)}$. 
 Thus, $\TT_{\GG}$ is not balanced. \hfill $\square$

%---SECTION: An Application to Algebraic Statistics------------
\section{Algebraic consequences for toric models}
% The algebraic perspective.
% Algebraic consequences
\label{sec: application}

A major endeavour in algebraic statistics is to identify which statistical models can be implicitly defined as the intersection of the space of model parameters with an algebraic variety \cite{S19} and, for such models, identify when the defining algebraic variety admits special properties.  
One commonly investigated question is: when is the defining algebraic variety toric?
In this case the model  is  called a \emph{toric model}.

A discrete DAG model $\MM(\GG)$ is toric when $\GG$ is a perfect DAG \cite[Proposition 3.28]{L96}.
Analogously, a staged tree model $\MM_{(\TT,\theta)}$ is toric if $(\TT,\theta)$ is a balanced staged tree \cite{DG20}. Theorem~\ref{thm: perfectly balanced and stratified}, establishes that the condition for being
toric for staged tree models is equivalent to the condition of being toric for DAG models when we restrict to staged trees in $\TT_{\mathbb{G}}$. This implies equality between certain ideals
associated to a discrete decomposable model. In this section we
explain the relation between these ideals and summarize them in
Corolllary~\ref{cor:ultimate truth}.
% Toric DAG models were studied in \cite{GMS06}, and later toric staged tree models were studied in .
% Both approaches to the problem identify toric models by showing that an ideal of a variety defining the model coincide with a toric ideal.  

Given a discrete random vector $X_{[p]}$ with state space $\RR$ and a DAG $\GG = ([p],E)$, we define two collections of indeterminates: $D = \{p_{\xx}: \xx\in\RR\}$ and $U = \{q_{i;x_i;\xx_{\pa_\GG(i)}}: i\in[p], x_i\in\RR_{\{i\}},\xx_{\pa_\GG(i)}\in\RR_{\pa_\GG(i)}\}$. 
It follows from equation~\eqref{eqn: DAG factorization} that the 
Zariski closure of $\MM(\GG)$ is the algebraic variety defined by  the vanishing of the kernel of the map of polynomial rings
\[
\Phi_\GG: \R[D]\rightarrow\R[U]/\mathfrak{q}; \qquad \Phi_\GG: p_{\xx}\mapsto \prod_{i\in[p]}q_{i;x_i;\xx_{\pa_\GG(i)}},
\]
where 
\[
\mathfrak{q} = \langle 1-\sum_{x_i\in\RR_{\{i\}}}q_{i;x_i;\xx_{\pa_\GG(i)}}: i\in[p],\xx_{\pa_\GG(i)}\in\RR_{\pa_\GG(i)}\rangle.
\]
That is, $\MM(\GG)=\Delta_{|\RR|-1}^{\circ}\cap V(\ker(\Phi_\GG))$, we refer to \cite{GSS05} for more details about the defining
ideal of $\MM(\GG)$.
Let $\tilde\GG$ denote the skeleton (i.e. the underlying undirected graph of $\GG$), and $\CC_{\tilde\GG}$ its set of maximal cliques.
We define another the set of indeterminates $W = \{\phi_{\xx_C} : C\in\CC_{\tilde\GG}, \xx\in\RR\}$, and a second map of polynomial rings
\[
\Phi_{\tilde\GG}: \R[D]\rightarrow\R[W]; \qquad\Phi_{\tilde\GG}: p_{\xx}\mapsto\prod_{C\in\CC_{\tilde\GG}}\phi_{\xx_C},
\]
whose kernel is a toric ideal.
It follows from \cite[Theorem 4.4]{GMS06} that $\ker(\Phi_\GG) = \ker(\Phi_{\tilde\GG})$ whenever $\GG$ is perfect, and hence $\MM(\GG)$ is toric. 

Given a staged tree $(\TT,\theta)$ with $\TT = (V,E)$ and labeling $\theta: E\rightarrow \LL$, we  define the polynomial ring $\R[z,\LL]$ and the ideal 
$\mathfrak{q}' = \langle 1-\sum_{e\in E(v)}\theta(v) : v\in V\rangle$. 
In \cite{DG20}, the authors showed that the Zariski closure of $\MM_{(\TT,\theta)}$ is defined by the vanishing of the kernel of the map 
\[
\Psi_\TT: \R[D]\rightarrow\R[z,\LL]/\mathfrak{q}';
\qquad
\Psi_\TT: p_\xx \mapsto z\cdot\prod_{e\in\lambda(v)}\theta(e).
\]
They also considered the kernel of the toric map
\[
\Psi_\TT^\toric: \R[D]\rightarrow\R[z,\LL]; \qquad \Psi_\TT^\toric: p_\xx \mapsto z\cdot\prod_{e\in\lambda(v)}\theta(e),
\]
and they showed that $\TT$ is balanced if and only if $\ker(\Phi_\TT) = \ker(\Psi_\TT^\toric)$.
Hence, $\MM_{(\TT,\theta)}$ is toric whenever $\TT$ is balanced.  

While both the result of \cite{GMS06} and \cite{DG20} show that perfect DAG models and balanced staged tree models, respectively, are toric via a coincidence of ideals, it is not a priori clear that the coincidence of ideals $\ker(\Psi_{\TT_\GG}) = \ker(\Psi_{\TT_\GG}^\toric)$ implies the coincidence of ideals $\ker(\Phi_\GG) = \ker(\Phi_{\tilde\GG})$; i.e., that the identified toric staged tree models generalize the identified toric DAG models. %i.e., that the method of proof of \cite{DG20} is a direct generalization of that of \cite{GMS06}.  
The next corollary establishes this to be the case.
\begin{corollary}
\label{cor:ultimate truth} Let $\GG$ be a DAG and $\TT_{\GG}$ its
staged tree representation with respect to some linear extension.
The following are equivalent:
\begin{enumerate}
    \item $\GG$ is perfect,
    
    \item $\ker(\Phi_{\GG})=\ker(\Phi_{\tilde{\GG}})$, and 
    
    \item $\ker(\Psi_{\TT_\GG})=\ker(\Psi_{\TT_\GG}^{\toric})$.
    \end{enumerate}
    If any of $(1),(2)$ or $(3)$ hold, then $\ker(\Phi_{\tilde\GG}) = \ker(\Phi_\GG) =\ker(\Psi_{\TT_\GG}) = \ker(\Psi_{\TT_\GG}^\toric)$.
\end{corollary}
\begin{proof}
$(1)\Leftrightarrow (2)$: This follows from \cite[Theorem 4.4]{GMS06}. $(1) \Leftrightarrow (3)$: By an application of Theorem~\ref{thm: perfectly balanced and stratified} and \cite[Theorem 3.1]{DG20}, we see that $\GG$ is perfect if and only if $\ker(\Psi_{\TT_\GG}) = \ker(\Psi_{\TT_\GG}^\toric)$. 
\end{proof}

% Hence, the coincidence of ideals $\ker(\Phi_\GG) = \ker(\Phi_{\tilde\GG})$ occurs only if the coincidence of ideals $\ker(\Psi_{\TT_\GG}) = \ker(\Psi_{\TT_\GG}^\toric)$ also occurs.  
% In which case, we have that 

% That is, by Theorem~\ref{thm: perfectly balanced and stratified}, the coincidence of ideals $\ker(\Psi_{\TT_\GG}) = \ker(\Psi_{\TT_\GG}^\toric)$ implies the coincidence of the ideals $\ker(\Phi_{\tilde\GG}) = \ker(\Phi_\GG)$.
% In view of Corollary~\ref{cor:ultimate truth}, Theorem~\ref{thm: perfectly balanced and stratified} implies that the toric staged tree models identified in \cite{DG20}.
%  genuinely generalize the toric DAG models identified in \cite{GMS06}.
%  but also that the method of proof of \cite{DG20} for showing that balanced models are toric generalizes the method of proof of \cite{GMS06} for showing that perfect DAG models are toric.

%---Possible Statistical Applications
\section{Future Directions in Statistics}
\label{sec: statistical applications}
The family of decomposable models plays an important role in probabilistic inference via DAGs.  
When one wishes to answer a probabilistic query such as, ``What is $P(X_i)$?'' given data drawn from a joint distribution $P(X_1,\ldots,X_p)$ Markov to a DAG $\GG$, the standard approach is to identify a chordal covering for the DAG (i.e., a chordal graph constructed by adding additional edges to the skeleton of $\GG$), and then form a clique tree which can be used to dynamically answer the query \cite{KF09}.  The nodes of the clique tree corresponds to cliques in the chordal covering of $\GG$, and the number of nodes in each of these cliques gives an upper bound on the complexity of answering the probabilistic query.  Hence, it is best to identify a chordal covering of the DAG that has smallest possible clique sizes.  The complexity bound for probabilistic inference via clique trees for the given DAG, called the \emph{treewidth} of the DAG, is defined to be one less than the size of a maximal clique in such a minimal chordal covering.  

It follows that decomposable models (i.e, chordal graphs) with small treewidth are desirable from the perspective of complexity of probabilistic inference.  In this work, we have established the balanced staged trees as both a combinatorial and algebraic generalization of decomposable models.  It would therefore be interesting to know if this generalization also generalizes the nice statistical properties of decomposable models in regards to probabilistic inference; that is, do balanced staged trees play the same role as decomposable models when conducting probabilistic infererence in context-specific settings via staged trees?  Does the notion of treewidth naturally generalize to balanced staged tree models in such as way as to offer analogous complexity bounds for context-specific probabilistic inference?  Exploring such questions would be very natural statistical follow-up work to the result of Theorem~\ref{thm: perfectly balanced and stratified}.

\bigskip

\noindent
{\bf Acknowledgements}. 
Liam Solus was supported a Starting Grant (No. 2019-05195) from Vetenskapsr\aa{}det (The Swedish Research Council), and by the Wallenberg AI, Autonomous Systems and Software Program (WASP) funded by the Knut and Alice Wallenberg Foundation. Eliana Duarte was supported by the Deutsche Forschungsgemeinschaft DFG under grant 314838170, GRK 2297 MathCoRe.

%---BIBLIOGRAPHY---------


\begin{thebibliography}{99}

\bibitem{AD19}
	L.~Ananiadi and E.~Duarte. 
	\emph{Gr\" obner bases for staged trees.}
	Algebraic Statistics 12 (2021): 1-20.

%\bibitem{BFGK96}
%	C.~Boutilier, N.~Friedman, M.~Goldszmidt, and D.~Koller. 
%	\emph{Context-specific independence in Bayesian networks.}
%	In Proceedings of the 12th Conference on Uncertainty in Artificial Intelligence, 115?123 (1996).
%	
%\bibitem{C95}
%	D.~M.~Chickering.
%	\emph{A transformational characterization of equivalent Bayesian network structures.}
%	Proceedings of the Eleventh Conference on Uncertainty in Artificial Intelligence. Morgan Kaufmann Publishers Inc., (1995).

% \bibitem{C02}
% 	D.~M.~Chickering.
% 	\emph{Optimal structure identification with greedy search.}
% 	Journal of machine learning research 3.Nov (2002): 507-554.

%\bibitem{CHM97}
%	D.~M.~Chickering, D.~Heckerman, and C.~Meek.
%	\emph{A Bayesian approach to learning Bayesian networks with local structure.}
%	The Proceedings of the Thirteenth conference on Uncertainty in artificial intelligence (UAI'97). Morgan Kaufmann Publishers Inc., San Francisco, CA, USA, 80-89 (1997).

\bibitem{CGS18}
	Collazo~R.A., C.~G\"orgen, and J.Q.~Smith.
	\emph{Chain event graphs.}
	CRC Press, 2018.
	
% \bibitem{CLO13}
%   D.~Cox, J.~Little and D.~ O'Shea.
%   \emph{Ideals, varieties, and algorithms: an introduction to computational algebraic geometry and commutative 
%   algebra.}
%   Springer Science \& Business Media, 2013.
   
    
% \bibitem{DSS08}
%     M.~Drton. B.~Sturmfels and S.~Sullivant. 
%     \emph{Lectures on algebraic statistics.} 
%     Vol. 39. Springer Science \& Business Media, 2008.

\bibitem{DG20}
	E.~Duarte and C.~G\"orgen. 
	\emph{Equations defining probability tree models.}
	Journal of Symbolic Computation 99 (2020): 127-146.

% \bibitem{E20}
%     R.~Evans
%     \emph{Model selection and local geometry.}
%     Annals of Statistics. Forthcoming .(2020)
%     Preprint available at \url{https://arxiv.org/pdf/1801.08364.pdf} (2020). 
    
% \bibitem{EGS05}
% 	F.~Eberhardt, C.~Glymour, and R.~Scheines. 
% 	\emph{On the number of experiments sufficient and in the worst case necessary to identify all causal relations among n variables.}
% 	In the Proceedings of the 22nd Conference on Uncertainty in Artificial Intelligence, pp. 178-184, 2005.
	
% \bibitem{FG65}
%     D.~Fulkerson and O.~Gross.
%     \emph{Incidence matrices and interval graphs.}
%     Pacific journal of mathematics 15.3 (1965): 835-855.
    
%\bibitem{FG96}
%	N.~Friedman and M.~Goldszmidt. 
%	\emph{Learning Bayesian networks with local structure.}
%	In Proceedings of the 12th Conference on Uncertainty in Artificial Intelligence, 252-262 (1996).

\bibitem{GHKM01}
	D.~Geiger, D.~Heckerman, H.~King, C.~Meek. 
	\emph{Stratified exponential families: Graphical models.}
	%Artificial Intelligence 82.1-2 (1996): 45-74.
	Annals of Statistics 29(2): 505-529, 2001
	
% \bibitem{GH96}
% 	D.~Geiger and D.~Heckerman. 
% 	\emph{Knowledge representation and inference in similarity networks and Bayesian multinets.}
% 	Artificial Intelligence 82.1-2 (1996): 45-74.
	
\bibitem{GMS06}
	D.~Geiger, C.~Meek, and B.~Sturmfels.
	\newblock On the toric algebra of graphical models.
	\newblock {\em Ann. Statist.}, 34(3):1463-1492, 2006.

% \bibitem{M2}
% D.~Grayson and M.~Stillman.
% {\em  Macaulay2, a software system for research in algebraic geometry},
% available at \url{http://www.math.uiuc.edu/Macaulay2/}.

\bibitem{GSS05}
	L.~D.~Garcia, M.~Stillman, and B.~Sturmfels. 
	\emph{Algebraic geometry of Bayesian networks.}
	Journal of Symbolic Computation 39.3-4 (2005): 331-355.
	
% \bibitem{GS16}
% 	C.~G\"orgen and J.~Q.~Smith. 
% 	\emph{A differential approach to causality in staged trees.}
% 	Conference on Probabilistic Graphical Models , (2016), 207-25.

\bibitem{GS18}
	C.~G\"orgen and J.~Q.~Smith. 
	\emph{Equivalence classes of staged trees.}
	Bernoulli 24.4A (2018): 2676-2692.

%\bibitem{GBRS18}
%	C.~G\"orgen, A.~Bigatti, E.~Riccomagno, and J.~Q.~Smith. 
%	\emph{Discovery of statistical equivalence classes using computer algebra}
%	International Journal of Approximate Reasoning 95 (2018):167-184

% \bibitem{H91}
% 	D.~Heckerman.
% 	\emph{Probabilistic Similarity Networks.}
% 	MIT Press, Cambridge, MA (1991).

% \bibitem{HS14}
%     J.~Huh and B.~Sturmfels: Likelihood geometry,
%     in {\em Combinatorial Algebraic Geometry}, pages 63--117. Springer,
%     2014.
	
% \bibitem{KRS19}
%   T.~Kahle, J.~ Rauh, and S.~Sullivant.
%   \emph{Algebraic aspects of conditional independence and graphical models.}
%   In {\em Handbook of graphical models}, Chapman \& Hall/CRC Handb.
%   Mod. Stat. Methods, pages 61--80. CRC Press, Boca Raton, FL, 2019.

% \bibitem{KJSB19}
% 	M.~Kocaoglu, A.~Jaber, K.~Shanmugam, and E.~Bareinboim.
% 	\emph{Characterization and learning of causal graphs with latent variables from soft interventions.}
% 	Advances in Neural Information Processing Systems. 2019.
	
\bibitem{KF09}
	D.~Koller and N.~Friedman.
	\emph{Probabilistic graphical models: principles and techniques.}
	MIT press, 2009.
	
\bibitem{L96}
	S.~L.~Lauritzen:
	 {\em Graphical Models}, 
	 Oxford University Press, 1996.
	 
%\bibitem{M13}
%	C.~Meek.
%	\emph{Strong completeness and faithfulness in Bayesian networks.}
%	arXiv preprint arXiv:1302.4973 (2013).
 
% \bibitem{P09}
% 	J.~Pearl.
% 	\emph{Causality.}
% 	Cambridge university press, 2009.
	
%\bibitem{PZ03}
%	D.~Poole and N.~Zhang. 
%	\emph{Exploiting contextual independence in probabilistic inference.} 
%	Journal of Artificial Intelligence Research {\bf18}, 263?313 (2003).

% \bibitem{SL14}
% 	K.~Sadeghi and S.~Lauritzen.
% 	\emph{Markov properties for mixed graphs.} 
% 	Bernoulli 20.2 (2014): 676-696.

\bibitem{SA08}
    J.~Q. Smith and P.~E. Anderson.
    \newblock Conditional independence and chain event graphs.
    \newblock {\em Artificial Intelligence}, 172(1):42 -- 68, 2008.

% \bibitem{So19}
% 	L.~Solus.
% 	\emph{Interventional Markov Equivalence for Mixed Graph Models.}
% 	arXiv preprint arXiv:1911.10114 (2019).
	
% \bibitem{SWU19}
% 	L.~Solus, Y.~Wang, and C.~Uhler.
% 	\emph{Consistency guarantees for permutation-based causal inference algorithms.}
% 	To appear: Biometrika. (2020). 

% \bibitem{SGS00}
%   P.~Spirtes, C.~Glymour, and R.~Scheines.  
%   \emph{Causation, Prediction, and Search}. MIT press, 2000.
	
%\bibitem{S96}
%	B.~Sturmfels.
%	\emph{Gr\"obner bases and convex polytopes.}
%	Vol. 8. American Mathematical Soc., 1996.

%\bibitem{S04}
%    B.~ Sturmfels. 
%    \emph{Solving systems of polynomial equations.} 
%    No. 97. American Mathematical Soc., 2002.

% \bibitem{S08}
% 	S.~Sullivant.
% 	\emph{Algebraic geometry of Gaussian Bayesian networks.}
% 	Advances in Applied Mathematics 40.4 (2008): 482-513.

\bibitem{S19}
	S.~Sullivant.
	\newblock {\em Algebraic statistics}, volume 194 of {\em Graduate Studies in Mathematics}.
	\newblock American Mathematical Society, Providence, RI, 2018.

% \bibitem{URBY13}
%     C.~Uhler, G.~Raskutti, P. B\"uhlman, and B.~ Yu
%     \emph{Geometry of the faithfulness assumption in causal inference}.
%     Annals of Statistics. 41(2) (2013): 436-463.

%\bibitem{VP90}
%	T.~Verma and J.~Pearl.
%	\emph{Causal networks: Semantics and expressiveness.}
%	Machine intelligence and pattern recognition. Vol. 9. North-Holland, 1990. 69-76.
%	
% \bibitem{VP92} 
% 	T.~Verma and J.~Pearl. 
% 	\emph{Equivalence and synthesis of causal models}.
% 	Proceedings of the Seventh International Conference on Uncertainty in Artificial Intelligence. Morgan Kaufmann Publishers Inc., 1991.

% \bibitem{WSYU17}
% 	Y.~Wang, L.~Solus, K.D.~Yang, and C.~Uhler.
% 	\emph{Permutation-based causal inference algorithms with interventions.}
% 	Advances in Neural Information Processing Systems. 2017.
	
% \bibitem{YKU18}
% 	K.D.~Yang, A.~Katcoff, and C.~Uhler.
% 	\emph{Characterizing and Learning Equivalence Classes of Causal DAGs under Interventions}
% 	Proceedings of the 35-th International Conference on Machine Learning (2018).
	
\end{thebibliography}
\end{document}